\documentclass[12pt,twoside,english,a4paper]{article}
\usepackage{babel,amssymb,amsthm}
\usepackage{graphicx}
\usepackage{amsmath}
\usepackage{bm}
\usepackage{float}
\usepackage{stmaryrd}

\newcommand{\smallfrac}[2]{{\textstyle\frac{#1}{#2}}} 
\newcommand{\jump}[1]{\llbracket #1\rrbracket}
\newcommand{\ave}[1]{\{\!\!\{#1\}\!\!\}}

\newcommand{\vertiii}[1]{{\left\vert\kern-0.25ex\left\vert\kern-0.25ex\left\vert #1 
    \right\vert\kern-0.25ex\right\vert\kern-0.25ex\right\vert}}

\newtheorem{proposition}{Proposition}[section]
\newtheorem{corollary}[proposition]{Corollary}
\newtheorem{lemma}[proposition]{Lemma}
\newtheorem{theorem}[proposition]{Theorem}

\numberwithin{equation}{section}

\title{A fully discrete BEM-FEM scheme \\ for transient acoustic waves}

\date{\today}

\author{Matthew E. Hassell
 \& Francisco--Javier Sayas\footnote{MEH and FJS partially funded by NSF grant DMS 1216356.}  \\
Department of Mathematical Sciences, University of Delaware, USA\\
{\tt \{mhassell,fjsayas\}@udel.edu }}

\newcommand{\where}[1]{\mbox{\rm (in $#1$)}}
\newcommand{\opp}[1]{\mathcal #1 *}

\begin{document}

\maketitle

\begin{abstract}
We study a symmetric BEM-FEM coupling scheme for the scattering of transient acoustic waves by bounded inhomogeneous anisotropic obstacles in a homogeneous field.  An incident wave in free space interacts with the obstacles and produces a combination of transmission and scattering.  The transmitted part of the wave is discretized in space by finite elements while the scattered wave is reduced to two fields defined on the boundary of the obstacles and is discretized in space with boundary elements.  We choose a coupling formulation that leads to a symmetric system of integro-differential equations.  The retarded boundary integral equations are discretized in time by Convolution Quadrature, and the interior field is discretized in time with the trapezoidal rule.  We show that the scattering problem generates a $C_0$ group of isometries in a Hilbert space, and use associated estimates to derive stability and convergence results.  We provide numerical experiments and simulations to validate our results and demonstrate the flexibility of the method. 
\end{abstract}

\section{Introduction}

In this paper we study the transmission and scattering of acoustic waves by inclusions in free space.   We focus on the case of a finite number of disjoint bounded, inhomogeneous and anisotropic inclusions.   An incident acoustic wave interacts with the inclusions, producing transmitted and scattered fields.   The wave transmitted through the inclusions is discretized in space with finite elements, while the scattered wave is reduced to two unknowns defined only on the boundary of the inclusions and is discretized in space with boundary elements.   For time discretization, we make use of trapezoidal rule based Convolution Quadrature (CQ) \cite{Lubich:1988a} and trapezoidal rule time stepping.  The scattered field can then be reconstructed from the boundary fields in a postprocessing step using Kirchhoff's formula.   By imposing two continuity conditions across the boundary of the inclusions, we arrive at a symmetric BEM-FEM coupling scheme.

There has been extensive work on the study of coupling of boundary and finite elements for steady-state and time-harmonic problems, but the literature on coupling schemes for transient problems is relatively sparse.  There are generally two types of coupling formulations, using one or two integral equations. The first ones (first analyzed by Johnson and N{\'e}d{\'e}lec \cite{JoNe:1980} for diffusion problems) lead to non-self adjoint formulations, while symmetric couplings (due to Costabel \cite{Costabel:1987} and Han \cite{Han:1990}) arrive at a symmetric system. Two-equation formulations are based on variational principles, and can be shown to always be stable, but at the cost of requiring all four of the operators of the Calder{\'o}n projector associated to the underlying PDE.  Single equation coupling methods are simpler, but do not have an underlying energy principle, and may therefore become unstable when there are large jumps in the material parameters.  The traditional two-equation coupling involves applying integral operators to the traces of finite element functions. There is an alternative formulation, using two fields on the boundary, that can keep the FEM and BEM modules better separated. For this work, we will study a two-equation, three-field coupling method.   Because we are using a two-equation formulation, we require all four retarded boundary integral operators associated to the wave equation.  

We next comment on some of the not very extensive existing literature on coupling of BEM and FEM for transient  waves.
The seminal paper \cite{BaBoPu:2001} provides several variational formulations of BEM-FEM coupling for time-dependent electromagnetic waves, with proofs of stability and convergence for their formulations, using a full-Galerkin treatment of the integral equations.  The papers \cite{AbJoRoTe:2011, BaLuSa:2015} deal with four-field formulations (two fields in the interior domain and two on the boundary) and aim at coupling an explicit interior time-stepping method with the retarded boundary integral equations on the boundary, differing in the use of Galerkin-in-time or CQ for the equations on the boundary. The papers \cite{FaMoSc:2012, FaMo:2014, FaMo:2015} contain successful computational studies of one-equation couplings, although a theoretical understanding of their stability is still missing. A preliminary semidiscrete stability analysis in the Laplace domain of the coupling method we will study here appears in \cite{LaSa:2009a}. 
In a similar vein, there is also recent work \cite{HsSaSa:2015} on the coupling of BEM and FEM for acoustic waves interacting with elastic media.

Traditional analysis of CQ discretizations of retarded integral equations has relied heavily on the use of the Laplace transform.   Precise bounds in terms of the Laplace parameter can be translated into estimates for the time-dependent problem.   The time domain estimates, however, are generally not sharp, because some regularity is lost by translating the problem to and from the Laplace domain.   In \cite{Lubich:1994} it is observed that the Laplace domain analysis can be avoided entirely, so that stability and convergence can be studied by directly considering the properties of the fully discrete (in space and time) solution to the underlying PDE.   This allows us to apply the theory of $C_0$ groups of isometries in Hilbert spaces to find sharper estimates than those provided by Laplace domain analysis.   Our analysis follows the first-order-in-space-and-time methodology proposed in \cite{SaHaQiSa:2015}.   By transforming the second-order-in-space-and-time wave equation into a first order system, we are able to circumvent a number of technical challenges that arise in the second-order-in-space-and-time analysis of \cite{Sayas:2014, Sayas:2013d}.

This paper is organized as follows. Section \ref{sec:2} prepares the basic notation and problem setting for the continuous and semidiscrete-in-space problems.  Section \ref{sec:3} introduces the first-order-in-space-and-time formulation and contains the analysis for the semidiscrete-in-space problem.   Through the application of a general result from the theory of semigroups of linear operators, we are able to establish stability for long times (with precise understanding of the growth of the energy in the system with respect to time) and optimal order of convergence for a Galerkin semidiscretization.  Section \ref{sec:4.1} carries out the analysis for the fully discrete problem when trapezoidal rule time stepping and trapezoidal rule CQ are used for time discretization.   We establish optimal order of convergence for the fully discrete scheme for data with sufficient regularity.   In Section \ref{sec:4.2} we include a detailed explanation of the algorithmic aspects of the coupling scheme.  Finally, section \ref{sec:5} provides numerical experiments and simulations.

\section{Continuous and semidiscrete problems}\label{sec:2}

\paragraph{Norms and inner products.} Given an open set $\mathcal O\subset \mathbb R^d$, we will denote the $L^2(\mathcal O)$ norm by $\|\cdot\|_{\mathcal O}$ and the $H^1(\mathcal O)$ norm by $\| \cdot\|_{1,\mathcal O}$. The inner products in $L^2(\mathcal O)$ and $L^2(\mathcal O)^d$ will be denoted $(\cdot,\cdot)_{\mathcal O}$. The $H^{\pm 1/2}(\Gamma)$ norms for a closed polygonal surface $\Gamma$ will be denoted $\|\cdot\|_{\pm 1/2,\Gamma}$. The duality product $H^{-1/2}(\Gamma)\times H^{1/2}(\Gamma)$ (with the spaces always in this order) will be denoted $\langle\cdot,\cdot\rangle_\Gamma$. 

\paragraph{Geometric setting and coefficients.}
Let $\Omega_j\subset \mathbb R^d$ ($j=1,\ldots,N$) be connected open sets lying on one side of their Lipschitz connected boundaries $\partial\Omega_j$ and such that their closures do not intersect. Let then $\Omega_-:=\cup_{j=1}^N \Omega_j$, $\Gamma:=\partial\Omega_-$, and $\Omega_+:=\mathbb R^d\setminus\overline{\Omega_-}$. In $\Omega_-$ we have two coefficients:
\[
\kappa:\Omega_-\to\mathbb R^{d\times d}_{\mathrm{sym}}, \qquad c:\Omega_-\to \mathbb R,
\]
where $\mathbb R^{d\times d}_{\mathrm{sym}}$ is the space of symmetric $d\times d$ real matrices. We assume that $c\in L^\infty(\Omega_-)$ and $c\ge c_0>0$ almost everywhere, so that $c^{-1}\in L^\infty(\Omega_-)$. We also assume that $\kappa \in L^\infty(\Omega_-)^{d\times d}$ is uniformly positive definite, i.e., there exists $\kappa_0>0$ such that
\[
\mathbf d\cdot (\kappa\mathbf d)\ge \kappa_0 |\mathbf d|^2, 
\qquad \forall \mathbf d\in \mathbb R^d,
\qquad \mbox{almost everywhere in $\Omega_-$}.
\]

\paragraph{Functional framework in the space variables.}
Before we state the transmission problem in a rigorous form we need to introduce some spaces and operators related to the space variables. The solution will take values in the spaces
\begin{subequations}
\begin{eqnarray}
H^1_\kappa(\Omega_-)
	&:=& \{ u\in H^1(\Omega_-)\,:\, \mathrm{div}\,(\kappa\nabla u)\in L^2(\Omega_-)\},\\
H^1_\Delta(\Omega_+)
	&:=& \{ u\in H^1(\Omega_+)\,:\, \Delta u\in L^2(\Omega_+)\}.
\end{eqnarray}
\end{subequations}
We will also need two trace operators $\gamma^\pm : H^1(\Omega_\pm)\to H^{1/2}(\Gamma)$ and the associated interior-exterior normal derivative operators $\partial_\nu^\pm:H^1_\Delta(\Omega_\pm)\to H^{-1/2}(\Gamma)$, defined in the usual weak form through Green's identities. We will also need the jump and average operators
\[
\jump{\gamma\cdot}:=\gamma^--\gamma^+,
	\quad
\ave{\gamma\cdot}:=\tfrac12(\gamma^-+\gamma^+),	
	\quad
\jump{\partial_\nu\cdot}:=\partial_\nu^--\partial_\nu^+,
	\quad
\ave{\partial_\nu\cdot}:=\tfrac12(\partial_\nu^-+\partial_\nu^+).
\]
For functions defined only in the interior domain $\Omega_-$ we will not use a superscript for the trace. We will also use the interior conormal derivative operator
$\partial_{\kappa,\nu}:H^1_\kappa(\Omega_-)\to H^{-1/2}(\Gamma)$.

\paragraph{Functional framework in the time variable.}
For the time variable we will use the language of vector-valued distributions. The test space $\mathcal D(\mathbb R)$ is the set of infinitely differentiable functions with compact support. This set is endowed with its usual topology \cite{Schwartz:1966}.
If $X$ is a Banach space, we say that $f\in \mathrm{TD}(X)$ when $f:\mathcal D(\mathbb R)\to X$ is a sequentially continuous linear function such that there exists a continuous causal functional
\[
g:\mathbb R\to X,
\qquad 
g(t)=0 \quad\forall t<0,
\qquad 
\| g(t)\|_X \le C t^m \quad m\ge 0, \quad t\ge 1
\]
and a non-negative integer $k$ satisfying
\begin{equation}
\langle f,\varphi\rangle=(-1)^k \int_{-\infty}^\infty g^{(k)}(\tau) \varphi(\tau)\mathrm d\tau
\qquad \forall \varphi \in \mathcal D(\mathbb R).
\end{equation}
This is equivalent to saying that $f$ is the $k$-th distributional derivative of a causal continuous polynomially bounded function. It is known \cite[Chapter 3]{Sayas:2014} that $f\in \mathrm{TD}(X)$ admits a distributional Laplace transform $\mathrm F$ defined in $\mathbb C_+:=\{ s\in \mathbb C\,:\,\mathrm{Re}\,s>0\}$, allowing for bounds of the form
\[
\| \mathrm F(s)\|_X \le C(\mathrm{Re}\,s) |s|^\mu \qquad \forall s\in \mathbb C_+,
\]
where $\mu\in \mathbb R$ and $C:(0,\infty)\to (0,\infty)$ is non-increasing and such that $C(\sigma)\le C \sigma^{-\ell}$ for some $\ell\ge 0$ as $\sigma\to 0$. Note that if $f\in \mathrm{TD}(X)$ and $A:X\to Y$ is linear and bounded ($A\in \mathcal B(X,Y)$), then $Af\in\mathrm{TD}(Y)$. Note also that distributional differentiation in the time variable is well defined in $\mathrm{TD}(X)$.

\paragraph{The transmission problem.} Let us assume that the incident wave is defined in a way such that
\[
\beta_0:=\gamma u^{\mathrm{inc}}\in\mathrm{TD}(H^{1/2}(\Gamma)),
\qquad
\beta_1:=\partial_\nu u^{\mathrm{inc}}\in\mathrm{TD}(H^{-1/2}(\Gamma)).
\]
This is a statement about `smoothness' of the incident wave in the space variables close to the boundary, as well as about causality of the traces of the incident wave. We look for
\begin{subequations}\label{eq:2.3}
\begin{equation}\label{eq:2.3o}
(u,u_+)\in \mathrm{TD}(H^1_k(\Omega_-))\times
\mathrm{TD}(H^1_\Delta(\Omega_+))
\end{equation}
  satisfying 
\begin{alignat}{6}
\label{eq:2.3a}
c^{-2}\ddot u &=\mathrm{div}(\kappa\nabla u) 
	&&\qquad &\where{L^2(\Omega_-)},\\
\ddot u_+ &= \Delta u_+ 
	&& &\where{L^2(\Omega_+)},\\
\label{eq:2.3c}
\gamma u &=\gamma^+ u_+ +\beta_0
	&& &\where{H^{1/2}(\Gamma)},\\
\label{eq:2.3d}
\partial_{\kappa,\nu} u &=\partial_\nu^+ u_++\beta_1
	&& &\where{H^{-1/2}(\Gamma)}.
\end{alignat}
\end{subequations}
Each of the equations in \eqref{eq:2.3} is satisfied as an equality of distributions taking values in the space in parentheses on the right-hand-side of the equation. We note that the vanishing initial conditions for $u$ are implicitly imposed by the condition \eqref{eq:2.3o}. Existence and uniqueness of solution to \eqref{eq:2.3} follows by taking Laplace transforms \cite[Section 6]{LaSa:2009a}.

\paragraph{Retarded potentials and associated integral operators.} The retarded layer potentials for the acoustic wave equation can be introduced using a uniquely solvable transmission problem. Let $\psi\in \mathrm{TD}(H^{1/2}(\Gamma))$ and $\eta\in \mathrm{TD}(H^{-1/2}(\Gamma))$. The problem that looks for $u\in \mathrm{TD}(H^1_\Delta(\mathbb R^d\setminus\Gamma))$ satisfying
\begin{subequations}\label{eq:2.4}
\begin{alignat}{6}
\ddot u=\Delta u
	&&\qquad &\where{L^2(\mathbb R^d\setminus\Gamma)},\\
\jump{\gamma u}=\psi,
	&&&\where{H^{1/2}(\Gamma)},\\
\jump{\partial_\nu u}=\eta,
	&& &\where{H^{-1/2}(\Gamma)},
\end{alignat}
\end{subequations}	
admits a unique solution, since it is a particular instance of \eqref{eq:2.3}. Using Laplace transforms and the theory of layer potentials for the resolvent operator of the Laplacian, it can be shown that there exist
\[
\mathcal D\in \mathrm{TD}(\mathcal B(H^{1/2}(\Gamma),H^1_\Delta(\mathbb R^d\setminus\Gamma))),
\qquad
\mathcal S\in \mathrm{TD}(\mathcal B(H^{-1/2}(\Gamma),H^1_\Delta(\mathbb R^d\setminus\Gamma))),
\]
such that the solution of \eqref{eq:2.4} can be written using the weak Kirchhoff formula (see \cite{LaSa:2009b} for a direct introduction to these operators in the three dimensional case)
\[
u=\mathcal S*\eta-\mathcal D*\psi.
\]
Here and in the sequel, the convolution symbol $*$ refers specifically to the convolution of a causal operator-valued distribution with a causal vector-valued distribution. The four retarded boundary integral operators are given by convolution with the averages of the Cauchy traces of the single and double layer retarded potentials:
\begin{alignat}{6}\label{eq:BIO}
\nonumber
\mathcal V :=\ave{\gamma \mathcal S} =\gamma^\pm \mathcal{S}&\in \mathrm{TD}(\mathcal B(H^{-1/2}(\Gamma),H^{1/2}(\Gamma))),\\
\nonumber
\mathcal K :=\ave{\gamma \mathcal D} &\in \mathrm{TD}(\mathcal B(H^{1/2}(\Gamma),H^{1/2}(\Gamma))),\\
\mathcal K^t :=\ave{\partial_\nu \mathcal S} &\in \mathrm{TD}(\mathcal B(H^{-1/2}(\Gamma),H^{-1/2}(\Gamma))),\\
\nonumber
\mathcal W :=-\ave{\partial_\nu \mathcal D} =-\partial_\nu^\pm \mathcal{D}&\in \mathrm{TD}(\mathcal B(H^{1/2}(\Gamma),H^{-1/2}(\Gamma))).\\
\nonumber
\end{alignat}
A fully detailed introduction to the retarded layer potentials and operators is given in \cite[Chapters 2 and 3]{Sayas:2014}, based on the Laplace domain analysis of Bamberger and HaDuong \cite{BaHa:1986a, BaHa:1986b}.

\paragraph{Boundary-field formulation.} Let $u_+$ be the exterior part of the solution of \eqref{eq:2.3} and let
\[
\phi:=\gamma^+ u_+,
\qquad
\lambda:=\partial_\nu^+ u_+.
\]
Then, by definition of the layer potentials and operators,
\begin{subequations}
\begin{alignat}{6}
u_+=\opp{D}\phi-\opp{S}\lambda
	&&\qquad&\where{H^1_\Delta(\Omega_+)},\\
\label{eq:2.5b}
\gamma^+ u_+=\tfrac12\phi+\opp{K}\phi - \opp{V}\lambda
	&&& \where{H^{1/2}(\Gamma)},\\
\label{eq:2.5c}
0=\opp{W}\phi+\tfrac12\lambda+\opp{K^t}\lambda
	&&&\where{H^{-1/2}(\Gamma)}.
\end{alignat}
\end{subequations}
The coupled boundary-field system consists of: (a) a variational-in-space formulation of \eqref{eq:2.3a} using \eqref{eq:2.3d}, (b) a non-local boundary condition obtained by substitution of \eqref{eq:2.3c} in \eqref{eq:2.5b}, and (c) the identity \eqref{eq:2.5c} to `symmetrize' the coupled system. We look for
\begin{subequations}\label{eq:2.6}
\begin{equation}
(u,\lambda,\phi)\in 
	\mathrm{TD}(H^1(\Omega_-))\times \mathrm{TD}(H^{-1/2}(\Gamma))
										\times\mathrm{TD}(H^{1/2}(\Gamma))
\end{equation}
satisfying
\begin{alignat}{6}
(c^{-2}\ddot u,w)_{\Omega_-}+(\kappa \nabla u,\nabla w)_{\Omega_-}
		&-\langle\lambda,\gamma w\rangle \hspace{3cm} \\
\label{eq:2.6b}
		&=\langle\beta_1,\gamma w\rangle \quad\forall w\in H^1(\Omega_-)
			 \quad &&\where{\mathbb R},\\
\gamma u+\opp{V}\lambda-\tfrac12\phi-\opp{K}\phi&=\beta_0
			\quad &&\where{H^{1/2}(\Gamma)},\\
\tfrac12\lambda+\opp{K^t}\lambda+\opp{W}\phi &=0
			\quad &&\where{H^{-1/2}(\Gamma)}.
\end{alignat}
\end{subequations}
The equivalence of the transmission problem with the boundary-field formulation \eqref{eq:2.6} is given in the next proposition. Its proof follows from taking Laplace transforms and using well-known results on integral representations of the solutions of elliptic equations \cite{McLean:2000}.

\begin{proposition}
Problem \eqref{eq:2.6} has a unique solution for arbitrary $\beta_0\in \mathrm{TD}(H^{1/2}(\Gamma))$ and $\beta_1\in \mathrm{TD}(H^{-1/2}(\Gamma))$. If $(u,\phi,\lambda)$ solves \eqref{eq:2.6} and $u_+=\opp{D}\phi-\opp{S}\lambda$, then $(u,u_+)$ is the unique solution of \eqref{eq:2.3}. Reciprocally, if $(u,u_+)$ is the solution of \eqref{eq:2.3} and $\phi:=\gamma^+ u_+$, $\lambda:=\partial_\nu^+ u_+$, then $(u,\lambda,\phi)$ is the solution of \eqref{eq:2.6}.
\end{proposition}

\paragraph{Semidiscretization in space.} We now introduce three finite dimensional subspaces
\[
U_h \subset H^1(\Omega_-),
	\qquad
X_h \subset H^{-1/2}(\Gamma),
	\qquad
Y_h \subset H^{1/2}(\Gamma).
\]
While we will keep Galerkin notation for the discretization of the variational equation \eqref{eq:2.6b}, we will follow \cite{LaSa:2009a} and shorten Galerkin semidiscrete-in-space equations on the boundary using polar spaces. If $\alpha\in \mathrm{TD}(H^{1/2}(\Gamma))$, we will write
\[
\alpha \in X_h^\circ \quad \where{H^{1/2}(\Gamma)}
\quad
\mbox{to denote}
\quad
\langle \mu^h,\alpha\rangle=0 \quad \forall\mu^h\in X_h \quad \where{\mathbb R}.
\]
Similary, if $\rho\in \mathrm{TD}(H^{-1/2}(\Gamma))$, we will write
\[
\rho \in Y_h^\circ \quad \where{H^{-1/2}(\Gamma)}
\quad
\mbox{to denote}
\quad
\langle \rho,\psi^h\rangle=0 \quad \forall\psi^h\in Y_h \quad \where{\mathbb R}.
\]
These conditions can also be described by taking Laplace transforms and imposing the respective tests with elements of $X_h$ and $Y_h$ to vanish for all values of the Laplace domain parameter $s$. We will also write conditions of the form
\begin{equation}\label{eq:2.7}
\eta\in X_h\quad \where{H^{-1/2}(\Gamma)} 
\qquad\mbox{and}\qquad
\psi\in Y_h\quad \where{H^{1/2}(\Gamma)}.
\end{equation}
For instance, if $\Pi_h^X:H^{-1/2}(\Gamma)\to X_h$ is the orthogonal projection onto $X_h$, the first condition in \eqref{eq:2.7} can be defined as $\Pi_h^X\eta=\eta$ as $H^{-1/2}(\Gamma)$-valued distrubutions.
The semidiscrete version of \eqref{eq:2.6} is the search for
\begin{subequations}\label{eq:2.8}
\begin{equation}
(u^h,\lambda^h,\phi^h)\in 
	\mathrm{TD}(U_h)\times \mathrm{TD}(H^{-1/2}(\Gamma))
										\times\mathrm{TD}(H^{1/2}(\Gamma))
\end{equation}
satisfying
\begin{equation}
\lambda^h\in X_h\quad \where{H^{-1/2}(\Gamma)},
\qquad
\phi^h\in Y_h\quad \where{H^{1/2}(\Gamma)}.
\end{equation}
and
\begin{alignat}{6}
\label{eq:2.8c}
(c^{-2}\ddot u^h,w^h)_{\Omega_-}+(\kappa \nabla u^h,\nabla w^h)_{\Omega_-}
		=\langle\lambda^h+\beta_1,\gamma w^h\rangle \quad\forall w^h\in U_h
			&\quad&&\where{\mathbb R},\\
\label{eq:2.8d}
\gamma u^h+\opp{V}\lambda^h-\tfrac12\phi^h-\opp{K}\phi^h-\beta_0\in X_h^\circ
			&&&\where{H^{1/2}(\Gamma)},\\
\label{eq:2.8e}
\tfrac12\lambda^h+\opp{K^t}\lambda^h+\opp{W}\phi^h\in Y_h^\circ
			&&&\where{H^{-1/2}(\Gamma)}.
\end{alignat}
\end{subequations}
A semidiscrete exterior solution is then defined with Kirchhoff's formula
\begin{equation}\label{eq:2.9}
u^\star =\opp{D}\phi^h-\opp{S}\lambda^h.
\end{equation}
In \eqref{eq:2.9} we have preferred not to name the output of the representation formula $u_+^h$ because we will be interested in this output as a distribution with values in $H^1_\Delta(\mathbb R^d\setminus\Gamma)$ instead of $H^1_\Delta(\Omega_+)$. Existence and uniqueness of solution to \eqref{eq:2.8} can be proved using the Laplace transform \cite[Section 6]{LaSa:2009a}. The technique relates the semidiscrete problem to an exotic transmission problem with two fields in the interior domain and one field in the exterior domain.

\begin{proposition}\label{prop:2.2}
Let $(u^h,\lambda^h,\phi^h)$ be the solution of \eqref{eq:2.8} and let $u^\star$ be defined by \eqref{eq:2.9}. The pair 
\begin{subequations}\label{eq:2.10}
\begin{equation}
(u^h,u^\star)\in \mathrm{TD}(U_h)\times \mathrm{TD}(H^1_\Delta(\mathbb R^d\setminus\Gamma))
\end{equation}
satisfies
\begin{alignat}{6}
\nonumber
(c^{-2}\ddot u^h, w^h)_{\Omega_-} +(\kappa \nabla u^h,\nabla w^h)_{\Omega_-} 
	+\langle\jump{\partial_\nu u^\star},\gamma w^h\rangle \qquad \\
	=\langle \beta_1,\gamma w^h\rangle\quad\forall w^h\in U_h
	&&\qquad& \where{\mathbb R},\\
\ddot u^\star=\Delta u^\star
	&&&\where{L^2(\mathbb R^d\setminus\Gamma)},\\
(\jump{\gamma u^\star},\jump{\partial_\nu u^\star}) \in Y_h \times X_h
	&&&\where{H^{1/2}(\Gamma) \times H^{-1/2}(\Gamma)},\\
(\partial_\nu^- u^\star, \gamma u^h-\gamma^+ u^\star-\beta_0) \in Y_h^\circ \times X_h^\circ
	&&&\where{H^{-1/2}(\Gamma) \times H^{1/2}(\Gamma)}.
\end{alignat}
\end{subequations}
Reciprocally, if $(u^h,u^\star)$ is the unique solution of \eqref{eq:2.10} and
\[
\phi^h=-\jump{\gamma u^\star}, \qquad 
\lambda^h=-\jump{\partial_\nu u^\star},
\] 
the triple $(u^h,\lambda^h,\phi^h)$ is the unique solution of \eqref{eq:2.8}.
\end{proposition}

\paragraph{Semidiscretization error.} To study the difference between the solutions of \eqref{eq:2.6} and \eqref{eq:2.8} we will use another exotic transmission problem. We first introduce the elliptic projection $\Pi_h^V:H^1(\Omega_-)\to U_h$ by solving the equations
\begin{equation}\label{eq:proj1}
(\kappa \nabla (\Pi_h^V u-u),\nabla w^h)_{\Omega_-}=0
\quad\forall w^h \in U_h,
\end{equation}
subject to the restrictions
\begin{equation}\label{eq:proj2}
\int_{\Omega_j} (\Pi_h^V u-u)=0 \quad j=1,\ldots,N,
\end{equation}
where $\Omega_j$ are the connected components of $\Omega_-$. 

\begin{proposition}
Let $(u,\lambda,\phi)$ and $(u^h,\lambda^h,\phi^h)$ be the respective solutions of \eqref{eq:2.6} and \eqref{eq:2.8} and let
\[
\varepsilon^h:=u^h-\Pi_h^V u, 
\qquad
\theta^h:=\Pi_h^V u-u,
\]
\[
\varepsilon^\lambda:=\lambda^h-\lambda,
\qquad
\varepsilon^\phi:=\phi^h-\phi,
\qquad
\varepsilon^\star:=u^\star-\opp{D}\phi+\opp{S}\lambda=
\opp{D}\varepsilon^\phi-\opp{S}\varepsilon^\lambda.
\]
Then
\begin{subequations}\label{eq:2.11}
\begin{equation}
(\varepsilon^h,\varepsilon^\star)\in \mathrm{TD}(U_h)\times\mathrm{TD}(H^1_\Delta(\mathbb R^d\setminus\Gamma))
\end{equation}
satisfies
\begin{alignat}{6}
(c^{-2}\ddot \varepsilon^h,w^h)_{\Omega_-} +(\kappa \nabla\varepsilon^h,\nabla w^h)_{\Omega_-}
	+\langle\jump{\partial_\nu \varepsilon^\star},\gamma w^h\rangle \qquad \\ 
		=-(c^{-2}\ddot\theta^h,w^h)_{\Omega_-}\qquad\forall w^h\in U_h
	&&\qquad& \where{\mathbb R},\\
\ddot \varepsilon^\star=\Delta \varepsilon^\star
	&&&\where{L^2(\mathbb R^d\setminus\Gamma)},\\
\gamma \varepsilon^h-\gamma^+ \varepsilon^\star+\gamma \theta^h \in X_h^\circ
	&&&\where{H^{1/2}(\Gamma)},\\
\jump{\gamma \varepsilon^\star}-\phi\in Y_h
	&&&\where{H^{1/2}(\Gamma)},\\
\jump{\partial_\nu \varepsilon^\star}-\lambda \in X_h
	&&&\where{H^{-1/2}(\Gamma)},\\
\partial_\nu^- \varepsilon^\star\in Y_h^\circ
	&&&\where{H^{-1/2}(\Gamma)}.
\end{alignat}
\end{subequations}
Reciprocally, if $(u,\lambda,\phi)$ is the solution of \eqref{eq:2.6}, $\theta^h:=\Pi_h^V u-u$, and $(\varepsilon^h,\varepsilon^\star)$ is the solution of \eqref{eq:2.11}, then $(u^h,\lambda^h,\phi^h)=(\varepsilon^h+\Pi_h^V u,\lambda-\jump{\partial_\nu \varepsilon^\star}, \phi-\jump{\gamma\varepsilon^\star})$ is the unique solution of \eqref{eq:2.8}.
\end{proposition}

\section{Analysis of an equivalent first order system}\label{sec:3}

\paragraph{Equivalent first order system.}
We will analyze problems \eqref{eq:2.10} and \eqref{eq:2.11} simultaneously. We thus look for 
\begin{subequations}\label{eq:3.1}
\begin{equation}
(u^h,u^\star) \in \mathrm{TD}(U_h)\times \mathrm{TD}(H^1_\Delta(\mathbb R^d\setminus\Gamma))
\end{equation}
satisfying
\begin{alignat}{6}
\nonumber
(c^{-2}\ddot u^h, w^h)_{\Omega_-} +(\kappa \nabla u^h,\nabla w^h)_{\Omega_-} 
	+\langle\jump{\partial_\nu u^\star},\gamma w^h\rangle \qquad \\
	=\langle \beta,\gamma w^h\rangle+(c^{-2} r,w^h)_{\Omega_-}\quad\forall w^h\in U_h
	&&\qquad& \where{\mathbb R},\\
\ddot u^\star=\Delta u^\star
	&&&\where{L^2(\mathbb R^d\setminus\Gamma)},\\
\gamma u^h-\gamma^+ u^\star-\alpha \in X_h^\circ
	&&&\where{H^{1/2}(\Gamma)},\\
\jump{\gamma u^\star}-\phi\in Y_h
	&&&\where{H^{1/2}(\Gamma)},\\
\jump{\partial_\nu u^\star}-\lambda \in X_h
	&&&\where{H^{-1/2}(\Gamma)},\\
\partial_\nu^- u^\star\in Y_h^\circ
	&&&\where{H^{-1/2}(\Gamma)},
\end{alignat}
\end{subequations}
for given data $\alpha,\beta,\lambda,\phi$ and $r$ taking values in the appropriate spaces. We will first transform \eqref{eq:3.1} into a first order system. To do that we introduce the antidifferentiation operator: given $f\in \mathrm{TD}(X)$, $\partial^{-1} f$ is the only element of $\mathrm{TD}(X)$ whose distributional derivative is $f$. The operator $\partial^{-1}$ is a weak version of
\[
(\partial^{-1} f)(t)=\int_0^t f(\tau)\mathrm d\tau.
\]
We will need the Sobolev space \cite{GiRa:1986}
\[
\mathbf H(\mathrm{div}, \mathbb R^d\setminus\Gamma):=
	\{ \mathbf v\in \mathbf L^2(\mathbb R^d):=L^2(\mathbb R^d)^d
		\,:\, \nabla\cdot\mathbf v\in L^2(\mathbb R^d\setminus\Gamma)\},
\]
endowed with its natural norm, which we will denote $\|\cdot\|_{\mathrm{div},\mathbb R^d\setminus\Gamma}$. For an element $\mathbf v$ of this space we can define the two sided normal components on $\Gamma$, $\gamma_\nu^\pm\mathbf v$ and the corresponding jump $\jump{\gamma_\nu\mathbf v}:=\gamma_\nu^-\mathbf v-\gamma_\nu^+\mathbf v$. We need finally the weighted orthogonal projection $P_h:L^2(\Omega_-)\to U_h$
\[
P_h r\in U_h, \qquad (c^{-2} (P_h r-r),w^h)_{\Omega_-}=0 \qquad \forall w^h \in U_h,
\]
a second discrete space
$
\mathbf V_h:=\nabla U_h=\{ \nabla u^h\,:\, u^h \in U_h\},
$
and the discrete operators $\mathrm{div}_h^\kappa:\mathbf L^2(\Omega_-)\to U_h$ and $\gamma_h^t: H^{-1/2}(\Gamma)\to U_h,$
given by the relation
\begin{equation}\label{eq:discop}
(c^{-2} (\mathrm{div}_h^\kappa \mathbf v+ \gamma_h^t \eta), w^h)_{\Omega_-}
	=-(\kappa\mathbf v,\nabla w^h)_{\Omega_-}+\langle \eta,\gamma w^h\rangle \qquad \forall w^h \in U_h.
\end{equation}
The first order formulation involves two new unknowns $\mathbf v^h:=\partial^{-1}\nabla u^h$ and $\mathbf v^\star:=\partial^{-1} \nabla u^\star$.
It looks for
\begin{subequations}\label{eq:3.2}
\begin{equation}
(u^h,u^\star,\mathbf v^h,\mathbf v^\star)
	\in \mathrm{TD}(U_h)\times \mathrm{TD}(H^1(\mathbb R^d\setminus\Gamma))
		\times \mathrm{TD}(\mathbf V_h)\times \mathrm{TD}(\mathbf H(\mathrm{div},\mathbb R^d\setminus\Gamma))
\end{equation}
satisfying
\begin{alignat}{6}\label{eq:3.2a}
\dot u^h = 
	\mathrm{div}_h^\kappa \mathbf v^h-\gamma_h^t \jump{\gamma_\nu \mathbf v^\star}
	+\gamma_h^t \beta	 + P_h r, 
	&\qquad&& \where{U_h},\\
\dot u^\star=\nabla\cdot\mathbf v^\star
	&&& \where{L^2(\mathbb R^d\setminus\Gamma)},\\
\dot{\mathbf v}^h=\nabla u^h
	&\qquad&& \where{\mathbf V_h},\\
\dot{\mathbf v}^\star=\nabla u^\star
	&&& \where{\mathbf L^2(\mathbb R^d\setminus\Gamma)},\\
\gamma u^h-\gamma^+ u^\star-\alpha \in X_h^\circ
	&&& \where{H^{1/2}(\Gamma)},\\
\jump{\gamma u^\star}-\phi \in Y_h
	&&& \where{H^{1/2}(\Gamma)},\\
\jump{\gamma_\nu \mathbf v^\star}-\partial^{-1}\lambda \in X_h
	&&& \where{H^{-1/2}(\Gamma)},\\
\gamma_\nu^-\mathbf v^\star \in Y_h^\circ
	&&& \where{H^{-1/2}(\Gamma)}.
\end{alignat}
\end{subequations}

\begin{proposition}
Problems \eqref{eq:3.1} and \eqref{eq:3.2} are equivalent.
\end{proposition}

\paragraph{An unbounded operator.} Consider the space
\[
\mathcal H:=U_h \times L^2(\mathbb R^d\setminus\Gamma) 
	\times \mathbf V_h \times \mathbf L^2(\mathbb R^d\setminus\Gamma),
\]
endowed with inner product whose associated norm is
\[
\| U\|_{\mathcal H}^2 
	= \| (u^h,u^\star,\mathbf v^h,\mathbf v^\star)\|_{\mathcal H}^2 
	:= \| c^{-1}u^h\|_{\Omega_-}^2 + \| u^\star\|_{\mathbb R^d\setminus\Gamma}^2
			+ \| \kappa^{1/2} \mathbf v^h\|_{\Omega_-}^2+ \|\mathbf v^\star\|_{\mathbb R^d\setminus\Gamma}^2.
\]
We also introduce the unbounded operator
\begin{equation}\label{eq:operatorA}
\mathcal A U=\mathcal A (u^h,u^\star,\mathbf v^h,\mathbf v^\star)
	:=(\mathrm{div}_h^\kappa\mathbf v_h-\gamma_h^t\jump{\gamma_\nu\mathbf v^\star},
		\nabla\cdot\mathbf v^\star,
		\nabla u^h,
		\nabla u^\star)
\end{equation}
defined in the domain $D(\mathcal A):=\mathcal U \times \boldsymbol{\mathcal V},$ 
where
\begin{eqnarray*}
\mathcal U &:=&
	\{ (u^h,u^\star)\in U_h\times H^1(\mathbb R^d\setminus\Gamma)\,:\,
			\gamma u^h-\gamma^+ u^\star\in X_h^\circ, \quad \jump{\gamma u^\star}\in Y_h\},\\
\boldsymbol{\mathcal V} &:=&
	\{ (\mathbf v^h,\mathbf v^\star)\in \mathbf V_h\times \mathbf H(\mathrm{div},\mathbb R^d\setminus\Gamma))\,:\,
			\jump{\gamma_\nu\mathbf v^\star}\in X_h,\quad \gamma_\nu^-\mathbf v^\star\in Y_h^\circ\}.
\end{eqnarray*}
For basic concepts of contractive $C_0$-semigroups of operators on Hilbert spaces (and the associated groups of isometries), we refer to \cite[Chapter 4]{Kesavan:1989} and the more comprehensive \cite{Pazy:1983}.

\begin{proposition}\label{prop:3.2}
The operators $\pm\mathcal A:D(\mathcal A)\subset\mathcal H \to \mathcal H$ are maximal dissipative. Therefore $\mathcal A$ is the inifinitesimal generator of a $C_0$-group of isometries in $\mathcal H$.
\end{proposition}

\begin{proof}
We first need to prove that
\begin{equation}\label{eq:3.4b}
(\mathcal A\,U,U)_{\mathcal H}=0 \qquad \forall U\in D(\mathcal A),
\end{equation}
which means, by definition, that $\pm\mathcal A$ are dissipative. To prove \eqref{eq:3.4b} we proceed as follows: given $U=
(u^h,u^\star,\mathbf v^h,\mathbf v^\star)\in D(\mathcal A)$,
\begin{eqnarray*}
(\mathcal A\,U,U)_{\mathcal H}
	&=& (c^{-2}(\mathrm{div}_h^\kappa\mathbf v_h-\gamma_h^t\jump{\gamma_\nu\mathbf v^\star}),u^h)_{\Omega_-}
		 +(\kappa\nabla u^h,\mathbf v^h)_{\Omega_-} \\
	& & + (\nabla\cdot\mathbf v^\star,u^\star)_{\mathbb R^d\setminus\Gamma}
		+(\nabla u^\star,\mathbf v^\star)_{\mathbb R^d\setminus\Gamma} \\
	& =& -\langle\jump{\gamma_\nu\mathbf v^\star},\gamma u^h\rangle 
			+\langle \gamma_\nu^-\mathbf v^\star,\gamma^- u^\star\rangle 
			-\langle \gamma_\nu^+\mathbf v^\star,\gamma^+ u^\star\rangle \\
	& = & -\langle\jump{\gamma_\nu\mathbf v^\star},\gamma^+ u^\star\rangle
			+\langle \gamma_\nu^-\mathbf v^\star,\gamma^+ u^\star\rangle 
			-\langle \gamma_\nu^+\mathbf v^\star,\gamma^+ u^\star\rangle = 0.
\end{eqnarray*}
We have applied: the definition of the discrete operators and the weak divergence theorem (definition of $\gamma^\pm_\nu$) in the second equality, and the transmission conditions included in the definitions $\mathcal U$ and $\boldsymbol{\mathcal V}$ for the third equality. 

To prove maximal dissipativity, we need to show that $\mathcal I\pm \mathcal A:D(\mathcal A)\to \mathcal H$ are surjective. We will only show the details for $\mathcal I-\mathcal A$, since the other case is essentially identical. Given $F=(f^h,f^\star,\mathbf g^h,\mathbf g^\star)\in \mathcal H$, we solve the coercive variational problem
\begin{subequations}
\begin{alignat}{6}
& (u^h,u^\star)\in \mathcal U \\
\nonumber
& (c^{-2} u^h,w^h)_{\Omega_-} + (\kappa \nabla u^h,\nabla w^h)_{\Omega_-}
	+(u^\star,w^\star)_{\mathbb R^d}+(\nabla u^\star,\nabla w^\star)_{\mathbb R^d\setminus\Gamma} \\
\label{eq:3.5b}
& \qquad = (c^{-2} f^h,w^h)_{\Omega_-} 
			- (\kappa\mathbf g^h,\nabla w^h)_{\Omega_-} 
			+ (f^\star,w^\star)_{\mathbb R^d}- (\mathbf g^\star,\nabla w^\star)_{\mathbb R^d\setminus\Gamma} \\
\nonumber
& \hspace{10cm} \forall (w^h,w^\star)\in \mathcal U,
\end{alignat}
\end{subequations}
and define
\begin{equation}\label{eq:3.6}
\mathbf v^h=\nabla u^h+\mathbf g^h, \qquad \mathbf v^\star=\nabla u^\star+\mathbf g^\star.
\end{equation}
If we test \eqref{eq:3.5b} with $(0,w^\star)\in \{0\}\times \mathcal D(\mathbb R^d\setminus\Gamma)\subset\mathcal U$ and substitute the second equation in \eqref{eq:3.6}, it follows that
\[
(u^\star,w^\star)_{\mathbb R^d\setminus\Gamma}+(\mathbf v^\star,\nabla w^\star)_{\mathbb R^d\setminus\Gamma}
	=(f^\star,w^\star)_{\mathbb R^d\setminus\Gamma} \qquad \forall w^\star \in \mathcal D(\mathbb R^d\setminus\Gamma).
\]
Therefore
\begin{equation}\label{eq:3.7}
u^\star=\nabla\cdot\mathbf v^\star+f^\star,
\end{equation}
which implies that $\mathbf v^\star\in \mathbf H(\mathrm{div},\mathbb R^d\setminus\Gamma).$ Substituting now \eqref{eq:3.6} and \eqref{eq:3.7} in \eqref{eq:3.5b}, we obtain
\begin{equation}\label{eq:3.9}
(c^{-2} u^h,w^h)_{\Omega_-}
	+(\kappa\mathbf v^h,\nabla w^h)_{\Omega_-}
	+(\nabla\cdot\mathbf v^\star,w^\star)_{\mathbb R^d\setminus\Gamma}
	+(\mathbf v^\star,\nabla w^\star)_{\mathbb R^d\setminus\Gamma}
	=(c^{-2} f^h,w^h)_{\Omega_-}
\end{equation}
for all $(w^h,w^\star)\in \mathcal U$. However, by the definition of the discrete operators \eqref{eq:discop} and the weak divergence theorem, we can equivalently (after some term rearrangement) write \eqref{eq:3.9} as 
\begin{eqnarray}\label{eq:3.10}
\nonumber
(c^{-2}(u^h-f^h-\mathrm{div}_h^\kappa\mathbf v^h+\gamma_h^t\jump{\gamma\mathbf v^\star}),w^h)_{\Omega_-}
\hspace{1cm}
	& & \\
+\langle\gamma_\nu^-\mathbf v^\star,\jump{\gamma w^\star}\rangle
+\langle\jump{\gamma_\nu \mathbf v^\star},\gamma^+ w^\star-\gamma w^h\rangle 
	&=&0 \quad\forall (w^h,w^\star)\in \mathcal U.
\end{eqnarray}
Let then $(\psi^h,\xi^h)\in Y_h\times X_h^\circ\subset H^{1/2}(\Gamma)^2$ and $w^h\in U_h$. We can choose $w^\star\in H^1(\mathbb R^d\setminus\Gamma)$ satisfying the trace conditions
$
\gamma^+ w^\star=\gamma w^h+\xi^h
$
and 
$\gamma^- w^\star=\gamma^+ w^\star+\psi^h.$
This proves that the operator
\[
\mathcal U\ni (w^h,w^\star) \longmapsto 
	(w^h,\jump{\gamma w^\star},\gamma^+ w^\star-\gamma w^h) \in U_h\times Y_h\times X_h^\circ
\]
is surjective. Therefore, \eqref{eq:3.10} is equivalent to
\begin{equation}\label{eq:3.11}
u^h=\mathrm{div}_h^\kappa\mathbf v^h-\gamma_h^t\jump{\gamma\mathbf v^\star}+f^h
\end{equation}
and the transmission conditions
\begin{equation}\label{eq:3.12}
\gamma_\nu^-\mathbf v^\star \in Y_h^\circ, \qquad \jump{\gamma_\nu \mathbf v^\star}\in X_h.
\end{equation}
These conditions imply that $(\mathbf v^h,\mathbf v^\star)\in \boldsymbol{\mathcal V}$. Therefore $U=(u^h,u^\star,\mathbf v^h,\mathbf v^\star)\in D(\mathcal A)$, and, finally, the collection of \eqref{eq:3.6}, \eqref{eq:3.7}, and \eqref{eq:3.11} implies that $U=\mathcal A U+F$. This finishes the proof of surjectivity of $\mathcal I-\mathcal A$.
\end{proof}

\paragraph{Lifting of the boundary conditions.} The next step is the construction of a lifting operator to move all non-homogeneities in the transmission conditions of \eqref{eq:3.2} (this includes the action of $\beta$ in the right-hand-side of \eqref{eq:3.2a}) to a right-hand-side of an operator equation $\dot U=\mathcal A U+F$. This operator is defined in Proposition \ref{prop:3.3}. Note that we do not give a bound for the norm of $\mathbf v^h$ because it will not be used in the sequel. The expression {\em $C$ is independent of $h$} will be used from this moment on to refer to a constant $C$ that is allowed to depend on parameters of the equation and on the geometry, but not on the choice of the three discrete subspaces involved.

\begin{proposition}\label{prop:3.3}
Given $(\varphi,\psi,\eta,\mu) \in H^{1/2}(\Gamma)^2\times H^{-1/2}(\Gamma)^2$, there exists a unique
\begin{subequations}\label{eq:3.13}
\begin{equation}
(u^h,u^\star,\mathbf v^h,\mathbf v^\star)\in 
	U_h\times H^1(\mathbb R^d\setminus\Gamma)\times \mathbf V_h \times \mathbf H(\mathrm{div},\mathbb R^d\setminus\Gamma)
\end{equation}
such that
\begin{alignat}{6}
u^h = 
	\mathrm{div}_h^\kappa \mathbf v^h-\gamma_h^t \jump{\gamma_\nu \mathbf v^\star}
	+\gamma_h^t \eta, & \qquad &
			 u^\star=\nabla\cdot\mathbf v^\star,\\
\mathbf v^h=\nabla u^h,& &
			\mathbf v^\star=\nabla u^\star,\\
\gamma u^h-\gamma^+ u^\star-\varphi \in X_h^\circ, & & 
			\jump{\gamma u^\star}-\psi \in Y_h,\\
\jump{\gamma_\nu \mathbf v^\star}-\mu \in X_h,& &
			\gamma_\nu^-\mathbf v^\star \in Y_h^\circ.
\end{alignat}
\end{subequations}
Furthermore, there exists $C>0$, independent of $h$, such that
\[
\| u^h\|_{1,\Omega_-}+\| u^\star\|_{1,\mathbb R^d\setminus\Gamma}
	+\|\mathbf v^\star\|_{\mathrm{div},\mathbb R^d\setminus\Gamma}
\le C (\|\varphi\|_{1/2,\Gamma}+\|\psi\|_{1/2,\Gamma}+\|\eta\|_{-1/2,\Gamma}+\|\mu\|_{-1/2,\Gamma}).
\]
\end{proposition}

\begin{proof}
Problem \eqref{eq:3.13} is equivalent to the problem that looks for 
\begin{subequations}\label{eq:3.14}
\begin{equation}
(u^h,u^\star)\in 
	U_h\times H^1_\Delta(\mathbb R^d\setminus\Gamma),
\end{equation}
satisfying
\begin{alignat}{6}
(c^{-2} u^h,w^h)_{\Omega_-}  
	+ (\kappa \nabla u^h,\nabla w^h)_{\Omega_-}+\langle\jump{\partial_\nu u^\star},\gamma w^h\rangle
	=\langle \eta,\gamma w^h\rangle & \quad & \forall w^h \in U_h,\\
			 u^\star=\Delta u^\star, & & \\
\gamma u^h-\gamma^+ u^\star-\varphi \in X_h^\circ, & & 
			\jump{\gamma u^\star}-\psi \in Y_h,\\
\jump{\partial_\nu u^\star}-\mu \in X_h,& &
			\partial_\nu^-u^\star \in Y_h^\circ,
\end{alignat}
\end{subequations}
and then computes $\mathbf v^h=\nabla u^h$ and $\mathbf v^\star=\nabla u^\star$. The variational formulation of \eqref{eq:3.14} is
\begin{subequations}\label{eq:3.15}
\begin{alignat}{6}
& (u^h,u^\star)\in U_h\times H^1(\mathbb R^d\setminus\Gamma),\\
& \jump{\gamma u^\star}-\psi\in Y_h, \qquad \gamma^+ u^\star-\gamma u^h-\varphi \in X_h^\circ,\\
& (c^{-2} u^h,w^h)_{\Omega_-}+ (\kappa \nabla u^h,\nabla w^h)_{\Omega_-} \\
\nonumber
& \hspace{1cm} +(u^\star,w^\star)_{\mathbb R^d}+(\nabla u^\star,\nabla w^\star)_{\mathbb R^d\setminus\Gamma}
	=\langle \eta-\mu,\gamma w^h\rangle+\langle\mu,\gamma^+w^\star\rangle \quad 
		\forall (w^h,w^\star)\in  \mathcal U
\end{alignat}
\end{subequations}
The solution of \eqref{eq:3.15} can be written as the sum $(0,u^\star_{nh})+(u^h,u^\star_0)$, where $\jump{\gamma u^\star}=\psi$, $\gamma^+ u^\star=\varphi$ and the pair $(u^h,u^\star_0)\in \mathcal U$ is the solution of a coercive variational problem in $\mathcal U$ with coercivity and boundedness constants independent of $h$. 
\end{proof}
\paragraph{An abstract theorem.} Before we state our main theorem, we prepare some notation.  For the proof, we refer the reader to \cite[Section 3]{SaHaQiSa:2015}.  Suppose that $\mathbb{H}, \mathbb{V},$ $ \mathbb{M}_1$, and $\mathbb{M}_2$ are Hilbert spaces, and that $\mathbb{V} \subset \mathbb{H}$ with continuous and dense embedding.  Let  $\mathsf{A}_\star : \mathbb{V} \rightarrow \mathbb{H}$ be a bounded linear operator such that the graph norm of $\mathsf{A}_\star$ is equivalent to the norm in the space $\mathbb{V}$.   Suppose $\mathsf{G}: \mathbb{M}_1 \rightarrow \mathbb H$ and $\mathsf{B} : \mathbb{V} \rightarrow \mathbb{M}_2$ are bounded linear operators.  Define the unbounded operator $\mathsf{A}:=\mathsf{A}_\star|_{D(\mathsf{A})} \subset \mathbb{H} \rightarrow \mathbb{H}$, where $D(\mathsf{A}) = \mathsf{Ker}(\mathsf{B})$.  We also assume $\pm \mathsf{A}$ are maximal dissipative operators.  We are then interested in the abstract differential equation
\begin{equation}\label{eq:ODE}
U \in \mathrm{TD}(\mathbb{H}), \qquad
\dot U = \mathsf A_\star U+\mathsf G \xi + F, \qquad
\mathsf BU =\chi,
\end{equation}
for data $(\xi,\chi) \in \mathrm{TD}(\mathbb{M}_1\times \mathbb{M}_2)$. The final hypothesis is related to the lifting of boundary conditions: we assume that the steady-state problem
\[
U\in \mathbb V, \qquad U-\mathsf A_\star U=\mathsf G \xi, \qquad \mathsf B U=\chi,
\]
has a unique solution for all $(\xi,\chi)\in \mathbb M_1\times \mathbb M_2$ and that there exists $C_{\mathrm{lift}}>0$ such that
\[
\| U\|_{\mathbb H}+\| U\|_{\mathbb V}\le C_{\mathrm{lift}} \| (\xi,\chi)\|_{\mathbb M_1\times \mathbb M_2}.
\]
We will also make use of the Sobolev spaces
\begin{alignat*}{6}
\mathcal{C}_+^k(X) &:= \{ f  \in \mathcal{C}^k(\mathbb{R}; X) ~:~ f(t) = 0 \quad t\leq 0\}, \\
W_+^k(X) &:=\{ f \in \mathcal{C}_+^{k-1}(\mathbb R; X) ~:~ f^{(k)} \in L^1(\mathbb{R};X), ~f^{(\ell)}(0) = 0 \quad \ell \leq k-1\}.
\end{alignat*}
Note that we have the inclusion $W_+^k(X) \subset \mathrm{TD}(X)$.  We then have the following theorem \cite{SaHaQiSa:2015}:
\begin{theorem}\label{the:3.1}
If $F \in W^1_+(\mathbb H)$ and $\Xi:=(\xi,\chi)\in W^2_+(\mathbb{M}_1 \times \mathbb{M}_2)$, then equation \eqref{eq:ODE} has a unique solution $U\in \mathcal C^1_+(\mathbb H)\cap \mathcal C_+(\mathbb V)$ and for all $t\ge 0$:
\begin{subequations}\label{eq:8}
\begin{align}
		\label{eq:8a}
\| U(t)\|_{\mathbb H} \le
	 & C_{\mathrm{lift}} \left( \int_0^t \| \Xi(\tau)\|_{\mathbb{M}_1 \times \mathbb{M}_2}\mathrm d\tau
							+ 2\int_0^t  \| \dot\Xi(\tau)\|_{\mathbb{M}_1 \times \mathbb{M}_2} 	\mathrm d\tau\right)
		+ \int_0^t \| F(\tau)\|_{\mathbb H}\mathrm d\tau,\\	
		\label{eq:8b}
\| \dot U(t)\|_{\mathbb H} \le
	& C_{\mathrm{lift}} \left( \int_0^t \| \dot\Xi(\tau)\|_{\mathbb{M}_1 \times \mathbb{M}_2}\mathrm d\tau
						+ 2\int_0^t  \| \ddot\Xi(\tau)\|_{\mathbb{M}_1 \times \mathbb{M}_2}\mathrm d\tau\right)
		+ \int_0^t \| \dot F(\tau)\|_{\mathbb H}\mathrm d\tau.
\end{align}
\end{subequations}  
\end{theorem}

\paragraph{Relationship to the problem at hand.} We will now explain how 
problem \eqref{eq:3.13} fits in this general abstract framework. The spaces are
\begin{alignat*}{6}
\mathbb{H} & := \mathcal H = \mathcal{U}_h \times L^2(\mathbb{R}^d\setminus\Gamma) \times \mathbf{V}_h \times \mathbf{L}^2(\mathbb{R}^d\setminus\Gamma),\\
\mathbb{V} & := \mathcal{U}_h \times H^1(\mathbb{R}^d\setminus\Gamma) \times \mathbf{V}_h \times \mathbf{H}(\mathrm{div};\mathbb{R}^d\setminus\Gamma),\\
\mathbb M_1 &:=H^{-1/2}(\Gamma), \qquad 
\mathbb M_2:= X_h^\ast \times (Y_h^\circ)^\ast \times (X_h^\circ)^\ast \times Y_h^\ast,
\end{alignat*}
where the asterisk is used to denote the dual space. The operator $\mathsf A_\star$ is given by the same expression as the operator $\mathcal A$ defined in \eqref{eq:operatorA}, but its domain is $\mathbb{V}$. The boundary conditions are taken care of by the operators
\[
\mathsf G \eta:=(-\gamma_h^t\eta, 0, 0, 0),
\qquad
\mathsf{B}U := ((\gamma u^h - \gamma^+ u^\star)|_{X_h}, \llbracket \gamma u^\star \rrbracket|_{Y_h^\circ}, \llbracket \gamma_\nu \mathbf{v}^\star \rrbracket|_{X_h^\circ}, \gamma_\nu^- \mathbf{v}^\star|_{Y_h}),
\]
where $\gamma_h^t$ is defined in \eqref{eq:discop}. We can understand what we mean by the various restrictions in $\mathsf{B}$ as follows.  Note that the difference in the traces $(\gamma u^h - \gamma^+ u^\star) \in H^{1/2}(\Gamma)=H^{-1/2}(\Gamma)^\ast$, and so we can recognize $(\gamma u^h - \gamma^+ u^\star)|_{X_h} : X_h \rightarrow \mathbb{R}$ as an element of $X_h^\ast$, defined by $X_h \ni \mu_h \mapsto \langle \mu_h, (\gamma u^h  - \gamma^+ u^\star)\rangle_\Gamma.$  The same explanation holds for the remaining components of $\mathsf{B}U$. The vector $\chi = (\alpha|_{X_h}, \phi|_{Y_h^\circ},\partial^{-1} \lambda|_{X_h^\circ}, 0)$ contains the transmission data.  Note that $D(\mathcal{A}) = \mathsf{Ker}(\mathsf{B})$ and $\mathsf{A}=\mathcal{A}$.  Finally $F=(P_h r,0,0,0)$.
We can now apply Theorem \ref{the:3.1} (the hypotheses have been verified in Propositions \ref{prop:3.2} and \ref{prop:3.3}) to problem \eqref{eq:3.2}. For convenience, we denote
$$
H_k(f,t | X) := \sum_{j=0}^k \int_0^t \| f^{(j)}(\tau) \|_X d\tau
$$
and $\mathbf H^{\pm1/2}(\Gamma):=(H^{\pm1/2}(\Gamma))^2$. 

\begin{proposition}\label{prop:3.5}
Let $\alpha, \phi \in W_+^2(H^{1/2}(\Gamma))$, $\beta,\lambda \in W_+^1(H^{-1/2}(\Gamma))$, and $r \in W_+^1(L^2(\Omega_-))$.  Then \eqref{eq:3.2} has a unique solution satisfying for all $t\geq 0$
\begin{alignat}{6}\label{eq:H1}
\nonumber
 \|c^{-1} u^h(t)\|_{\Omega_-} &+ \|\kappa \nabla u^h(t)\|_{\Omega_-} + \|u^\star(t)\|_{1, \mathbb R^d \setminus \Gamma} + \| \llbracket \gamma u^\star(t) \rrbracket \|_{1/2,\Gamma} \\
 	&  \leq  C \Big( H_2((\alpha,\phi), t | \mathbf{H}^{1/2}(\Gamma))+ H_2(\partial^{-1}(\beta,\lambda),t | \mathbf{H}^{-1/2}(\Gamma)) \\
\nonumber
	&  \hspace{.5in} + H_1(P_h r, t | L^2(\Omega_-))\Big),
\end{alignat}
where the constant $C$ does not depend on the time $t$ or $h$.
For $\alpha,\phi \in W^3_+(H^{1/2}(\Gamma))$, $\beta, \lambda \in W^2_+(H^{-1/2}(\Gamma))$, and $r\in W_+^2(L^2(\Omega_-))$
we have for all $t\geq 0$
\begin{alignat}{6}\label{eq:normal}
\nonumber
\| \llbracket \partial_\nu u^\star(t) \rrbracket \|_{-1/2,\Gamma} \leq C & \Big( H_2( (\dot{\alpha},\dot{\phi}), t | \mathbf{H}^{1/2}(\Gamma))+ H_2((\beta,\lambda),t | \mathbf{H}^{-1/2}(\Gamma))\\
	& + H_1(P_h \dot{r},t | L^2(\Omega_-))\Big).
\end{alignat}
\end{proposition}
\noindent
With this main result in hand, stability and semidiscretization error estimates follow as simple corollaries.
\begin{corollary}[Stability]\label{cor:3.1}
For data $\beta_0 \in W_+^2(H^{1/2}(\Gamma))$ and $\beta_1 \in W_+^1(H^{-1/2}(\Gamma))$ the semidiscrete scattering problem \eqref{eq:2.10} has a unique solution $(u^h,u^\star)$ such that
\begin{alignat*}{6}
 \|c^{-1} u^h(t)\|_{\Omega_-} + \|\kappa \nabla u^h(t)\|_{\Omega_-} + \|u^\star(t)\|_{1, \mathbb R^d \setminus \Gamma}+\| \phi^h(t) \|_{1/2,\Gamma} \\
 	& \hspace{-2in}  \leq  C \Big( H_2(\beta_0, t | H^{1/2}(\Gamma))+H_2(\partial^{-1}\beta_1,t | H^{-1/2}(\Gamma)) \Big).
\end{alignat*}
For $\beta_0 \in W_+^3(H^{1/2}(\Gamma))$ and $\beta_1 \in W_+^2(H^{-1/2}(\Gamma))$
we have the estimate
$$
\| \lambda^h(t) \|_{-1/2,\Gamma} \leq C \Big( H_2( \dot{\beta}_0, t | H^{1/2}(\Gamma))+H_2(\beta_1,t | H^{-1/2}(\Gamma))\Big).\\
$$
The constant $C$ is independent of $h$ and $t$.
\end{corollary}
\begin{proof}
We apply Proposition \ref{prop:3.5} with $\alpha = \beta_0\in W_+^2(H^{1/2}(\Gamma))$, $\beta = \beta_1 \in W_+^1(H^{-1/2}(\Gamma)),$ $\phi =0$, $\lambda =0$, and $r=0$.
\end{proof}

\begin{corollary}[Semidiscretization error]\label{cor:3.2}
Let $\Pi_h^X : H^{1/2}(\Gamma) \rightarrow X_h$ and $\Pi_h^Y : H^{-1/2}(\Gamma) \rightarrow Y_h$ be the orthogonal projections into the spaces $X_h$ and $Y_h$, respectively, and let $\Pi_h^V$ be the elliptic projection operator defined by \eqref{eq:proj1} and \eqref{eq:proj2}.  Suppose
\begin{equation*}
\phi \in W_+^{m}(H^{1/2}(\Gamma)), \quad \lambda \in W^{m-1}_+( H^{-1/2}(\Gamma)), \quad u\in W_+^m(H^1(\Omega_-)) \cap W_+^{m+1}(L^2(\Omega_-)).
\end{equation*}
If the above holds with $m=2$, then the Galerkin semidiscretization error $(\varepsilon^h, \varepsilon^\star):=(u^h-u, u^\star - \mathcal{D}\ast\phi + \mathcal{S}\ast \lambda)$ that solves equations \eqref{eq:2.11} satisfies for all $t \geq 0$
\begin{alignat*}{6}
\|c^{-1} \varepsilon^h(t) \|_{\Omega_-}+\|\kappa \nabla \varepsilon^h(t)\|_{\Omega_-} &+ \| \varepsilon^\star(t)\|_{1,\mathbb R^d \setminus\Gamma} + \|\phi^h(t) - \phi(t)\|_{1/2,\Gamma} \\
	   \leq C & \Big( H_2(u - \Pi_h^V u,t | H^1(\Omega_-)) + H_2(\phi - \Pi_h^Y\phi, t | H^{1/2}(\Gamma)) \\
	&   + H_2(\partial^{-1}(\lambda - \Pi_h^Y \lambda), t|H^{-1/2}(\Gamma))+H_1(\ddot{u}-\Pi_h^V \ddot{u},t|L^2(\Omega_-)) \Big).
\end{alignat*}
If the exact solution $(\lambda, \phi, u)$ satisfies the above with $m=3$, then we have the estimate
\begin{alignat*}{6}
\| \lambda^h(t) -\lambda(t)\|_{-1/2,\Gamma} \leq C & \Big(H_2(\dot{u} - \Pi_h^V\dot{u},t | H^1(\Omega_-)) + H_2(\dot{\phi} - \Pi_h^Y\dot{\phi}, t | H^{1/2}(\Gamma)) \\
	& + H_2(\lambda - \Pi_h^Y\lambda, t|H^{-1/2}(\Gamma))+H_1(\dddot{u}-\Pi_h^V \dddot{u},t|L^2(\Omega_-)) \Big).
\end{alignat*}
\end{corollary}
\begin{proof}
Note that the solution \eqref{eq:3.2} with $\alpha = 0$, $\beta=0$, $r=0$, $\phi = \Pi^Y_h \phi$, and $ \lambda = \Pi^X_h \lambda$ (i.e., the data $\phi$ and $\lambda$ take values in the discrete spaces) is the trivial solution.   If we now apply Proposition \ref{prop:3.5} with $\alpha= \gamma (u - \Pi_h^Vu)$, $\beta=0$, $r = \Pi^V_h \ddot{u} - \ddot{u}$, and $(\phi,\lambda) $ as in the hypotheses of the corollary, the result follows.
\end{proof}

\section{Full discretization}

For a full discretization, we make use of the trapezoidal rule based Convolution Quadrature (CQ) \cite{Lubich:1988a, Lubich:1994} for the solution of convolution equations and trapezoidal rule for solving the interior system of ODEs.  CQ was developed in the late 1980s by Christian Lubich as a method for the stable discretization for convolution equations.  It uses time domain readings of data and Laplace domain evaluations of the transfer function to produce time domain output.  An underlying ODE solver is used to carry out the time discretization, which can be any A-stable linear multistep method or an implicit Runge-Kutta method.   For a comprehensive introduction to the algorithmic aspects of CQ, see \cite{BaSc:2012, HaSa:2014}.   We present here a simple example to demonstrate the method.
\paragraph{A short introduction to CQ.}
Consider the causal convolution 
\begin{equation}\label{eq:4.1}
y(t) = \int_0^t f(t-\tau)g(\tau) d\tau
\end{equation}
where $y$ is unknown and $f$ and $g$ are known.  We will assume that we are using causal data (i.e. $f(t)$ and $g(t)$ are zero for $t<0$) and seek a causal solution $y$.  The function $f$ will be used through its Laplace transform $\mathrm{F}(s) = \mathcal{L}\{f(t)\}$.  We fix a uniform time step $k >0$ and a uniform time grid $t_n := nk$ for $n\geq 0$. CQ approximates the forward convolution \eqref{eq:4.1} by a discrete convolution
$$
y(t_n) = \sum_{m=0}^n \omega_m^\mathrm{F}(k)g(t_{n-m}),
$$
where the convolution weights $\omega_m^\mathrm{F}(k)$ are the coefficients of the Taylor series 
\begin{equation}\label{eq:4.1a}
\mathrm{F}\left(\frac{\delta(\zeta)}{k}\right) = \sum_{m=0}^\infty \omega_m^\mathrm{F}(k) \zeta^m.
\end{equation}
The function $\delta(\zeta)$ is called the transfer function for the CQ method, and is based on an underlying A-stable ODE solver.  In the case of the trapezoidal rule, the transfer function is $\delta(\zeta) = 2 \frac{1-\zeta}{1+\zeta}$.
CQ also can be used to solve convolution equations.  The continuous convolution equation (with $y$ still unknown) and its CQ discretization are
$$
g(t) = \int_0^t f(t-\tau)y(\tau)d\tau \quad\mbox{and}  \quad g(t_n) = \sum_{m=0}^n \omega_m^\mathrm{F}(k) y(t_{n-m}),
$$
respectively. The discrete convolution can can be written as a marching-on-in-time scheme
$$
\omega_0^\mathrm{F}(k)y(t_n) = g(t_n) - \sum_{m=1}^n \omega_m^\mathrm{F}(k)y(t_{n-m}).
$$
We note that there are many ways of implementing CQ, some of them using parallel computations at complex frequencies \cite{BaSa:2008, BaSc:2012, HaSa:2014}.

\subsection{Fully discrete analysis}\label{sec:4.1}

The fully discrete method consists of applying the trapezoidal rule based CQ to the semidiscrete equations \eqref{eq:2.8}.  Even if CQ, in practice, only produces solutions at discrete times, the method gives a theoretical extension of this solution to continuous time \cite{Lubich:1988a, HaSa:2014, Sayas:2014}.   The fully discrete solution will be denoted as $(u^h_k,\lambda^h_k,\phi^h_k)$.  The boundary solutions are then input to a CQ discretized Kirchhoff formula, outputting a field $u^\star_k$.   From the point of view of implementation (see also Section \ref{sec:4.2}) the monolithic application of CQ to the semidiscrete equations \eqref{eq:2.8} and to the representation formula \eqref{eq:2.9} is equivalent to the use of CQ for the retarded integral equations (\ref{eq:2.8d},~\ref{eq:2.8e}) and for the representation formula, coupled with a trapezoidal rule approximation of the linearly implicit second order differential equation \eqref{eq:2.8c} (see \cite[Proposition 12]{LaSa:2009a}).   An interesting feature of CQ applied to time domain boundary integral equations is the fact that the method is equivalent to applying the underlying ODE solver (in this case, the trapezoidal rule) to the evolutionary PDE satisfied by the potential post-processing.   This was already observed in \cite{Lubich:1994} and has been exploited for analysis in \cite{BaLaSa:2015} and \cite[Chapter 9]{Sayas:2014}.   In our case this will amount to carrying out the analysis directly on the variables $(u^h_k, u^\star_k)$.   

For the remaining analysis, we need to define the averaging and differencing operators
$$
\alpha_k g(t) := \frac{1}{2}\left( g(t) + g(t-k)\right), \quad \partial_k g(t) := \frac{1}{k}\left(g(t)-g(t-k)\right),
$$
and their squares
$$
\alpha_k^2 g(t)  = \frac{1}{4}\left( g(t) + 2 g(t-k) + g(t-2k)\right),\quad \partial_k^2 g(t) = \frac{1}{k^2} \left(g(t) -2g(t-k) + g(t-2k)\right).
$$
The fully discrete method looks for
\begin{equation}
(u^h_k,u^\star_k)\in \mathrm{TD}(U_h)\times \mathrm{TD}(H^1_\Delta(\mathbb R^d\setminus\Gamma))
\end{equation}
satisfying
\begin{subequations}\label{eq:4.3}
\begin{alignat}{6}
\nonumber
(c^{-2} \partial_k^2 u^h_k, w^h)_{\Omega_-} +( \alpha_k^2 \kappa \nabla u^h_k,\nabla w^h)_{\Omega_-} 
	+\langle\jump{ \alpha_k^2 \partial_\nu u^\star_k},\gamma w^h\rangle \qquad \\
	=\langle \alpha_k^2 \beta_1,\gamma w^h\rangle\quad\forall w^h\in U_h,\\
\partial_k^2 u^\star_k= \alpha_k^2 \Delta u^\star_k, \\
(\jump{\gamma u^\star_k},\jump{\partial_\nu u^\star_k}) \in Y_h\times X_h, \\
(\partial_\nu^- u^\star,\gamma u^h_k-\gamma^+ u^\star_k-\beta_0) \in Y_h^\circ\times X_h^\circ,
\end{alignat}
\end{subequations}
i.e., we have applied the trapezoidal rule to the second order differential equation \eqref{eq:2.10}.

\paragraph{Fully discrete error.} We define the consistency error for the trapezoidal rule time discretization for the interior and exterior fields by
$$
\chi_k^h := \partial_k^2 u^h - \alpha_k^2 \ddot{u}^h \qquad \text{and} \qquad \chi_k^\star:= \partial_k^2 u^\star - \alpha_k^2 \ddot{u}^\star.
$$
Subtracting equations \eqref{eq:4.3} from \eqref{eq:2.10} we find the error quantities $e_k^h := u^h - u_k^h$ and $e^\star_k:=u^\star - u_k^\star$ satisfy the error equations
\begin{subequations} \label{eq:4.4}
$$(e_k^h,e_k^\star) \in \mathrm{TD}(U_h) \times \mathrm{TD}(H^1_\Delta(\mathbb{R}^d\setminus\Gamma)) \\$$
\begin{alignat}{6}
\nonumber
\label{eq:4.4a}
(c^{-2} \partial_k^2 e_k^h,w^h)_{\Omega_-} + (\alpha_k^2 \kappa \nabla e_k^h,\nabla w^h)_{\Omega_-} + \langle \llbracket \partial_\nu \alpha_k^2 e_k^\star \rrbracket,\gamma w^h \rangle_\Gamma\qquad \\
		= (c^{-2}\chi_k^h,w^h)_{\Omega_-} 	\quad \forall w^h \in U_h, \\
\label{eq:4.4b}
		\partial_k^2 e^\star_k = \alpha_k^2 \Delta e_k^\star + \chi_k^\star, \\
\label{eq:4.4c}
	(\jump{\gamma e_k^\star},\jump{\partial_\nu e_k^\star}) \in Y_h\times X_h, \\
\label{eq:4.4d}
	(\partial_\nu^- e_k^\star,\gamma e_k^h-\gamma^+ e_k^\star) \in Y_h^\circ\times X_h^\circ.
\end{alignat}
\end{subequations}

Before we state the main theorem, we require the following lemma.
\begin{lemma}\label{lem:4.1}
If $e_k^\star$ is a continuous function of $t$ then the following Green's Identity holds for all $t\geq 0$:
\begin{equation*}
(\Delta e_k^\star(t),w^\star)_{\mathbb{R}^d\setminus \Gamma} + (\nabla e_k^\star(t), \nabla w^\star)_{\mathbb{R}^d\setminus\Gamma} =  \langle \llbracket \partial_\nu e_k^\star(t) \rrbracket, \gamma w^h \rangle_\Gamma \quad  \forall (w^h,w^\star) \in \mathcal{U}.
\end{equation*}
\end{lemma}
\begin{proof}
The following chain of equalities
\begin{alignat*}{6}
(\Delta e_k^\star,w^\star)_{\mathbb{R}^d\setminus \Gamma} + &
			(\nabla e_k^\star(t), \nabla w^\star)_{\mathbb{R}^d\setminus\Gamma} - \langle \llbracket \partial_\nu e_k^\star(t) \rrbracket, \gamma w^h \rangle_\Gamma \\
		& = \langle \partial_\nu^- e_k^\star(t), \gamma^- w^\star \rangle_\Gamma - \langle \partial_\nu^+ e_k^\star(t),\gamma^+ w^\star \rangle_\Gamma - \langle \llbracket \partial_\nu e_k^\star(t) \rrbracket, \gamma w^h \rangle_\Gamma  \\
		& =  \langle \llbracket \partial_\nu e_k^\star(t) \rrbracket,  \gamma^+ w^\star - \gamma w^h \rangle_\Gamma + \langle \partial_\nu^- e_k^\star(t) , \llbracket \gamma w \rrbracket \rangle_\Gamma = 0
\end{alignat*} 
holds for all $(w^h,w^\star)\in \mathcal{U}$.
\end{proof}
\begin{theorem}\label{th:4.2}
Suppose that $\beta_0 \in W_+^6(H^{1/2}(\Gamma))$ and $\beta_1 \in W_+^5(H^{-1/2}(\Gamma)).$ Then the natural error quantities $\widehat{e}_k^h :=\alpha_k e_k^h$, $\widehat{f}_k^h :=\partial_k e_k^h$, $\widehat{e}_k^\star := \alpha_k e_k^\star$, $\widehat{f}_k^\star := \partial_k e_k^\star$ for a trapezoidal rule in time discretization of \eqref{eq:2.10} satisfy for all $t\geq 0$
\begin{alignat}{6}\label{eq:4.5}
\nonumber
\| \widehat{f}_k^h(t)\|_{\Omega_-} +\| \kappa \nabla \widehat{e}_k^h(t) \|_{\Omega_-} + \| \widehat{f}_k^\star(t)\|_{\mathbb R^d \setminus \Gamma} + \| \nabla \widehat{e}_k^\star(t) \|_{\mathbb R^d \setminus \Gamma}\\
 		& \hspace{-2in}\leq C k^2 t \left( H_3(\beta^{(3)}_0,t| H^{1/2}(\Gamma)) + H_2(\beta^{(3)}_1, t| H^{-1/2}(\Gamma))\right).
\end{alignat}
We also have the $L^2$ error estimate
\begin{equation}\label{eq:4.6}
\| e_k^h(t) \|_{\Omega_-} + \| e^\star_k(t)\|_{\mathbb R^d \setminus \Gamma} \leq C k^2 t^2 \left(H_3(\beta^{(3)}_0, t| H^{1/2}(\Gamma)) + H_2(\beta^{(3)}_1, t| H^{-1/2}(\Gamma))\right).
\end{equation}
The error for $\llbracket \gamma e^\star_k \rrbracket=\phi^h_k - \phi^h$ is bounded as 
\begin{equation}\label{eq:4.7}
\| \alpha_k^2 \llbracket \gamma e_k^\star(t) \rrbracket \|_{1/2, \Gamma} \leq C k^2 \max \{ t, t^2 \} \left( H_3(\beta^{(3)}_0,t | H^{1/2}(\Gamma)) + H_2(\beta^{(3)}_1, t | H^{-1/2}(\Gamma)) \right).
\end{equation}
For $\beta_0\in W_+^7(H^{1/2}(\Gamma))$ and $\beta_1 \in W_+^6(H^{-1/2}(\Gamma))$, the error for $ \llbracket \partial_\nu e^\star_k \rrbracket  = \lambda^h_k - \lambda^h$ is bounded as
\begin{equation}\label{eq:4.8}
\| \alpha_k^2\llbracket \partial_\nu  e^\star_k(t) \rrbracket \|_{-1/2,\Gamma} \leq C k^2 \max\{1,t\} \left( H_3(\beta_0^{(4)},t | H^{1/2}(\Gamma)) + H_2(\beta_1^{(4)},t | H^{-1/2}(\Gamma)) \right).
\end{equation}
\end{theorem}
\begin{proof}
Using the definition of the hatted variables, \eqref{eq:4.4a}, \eqref{eq:4.4b}, and Lemma \ref{lem:4.1} it follows that 
\begin{alignat*}{6}
\partial_k\left( \widehat{f}_k^h (t),w^h\right)_{\Omega_-}+\partial_k \left( \widehat{f}_k^\star(t),w^\star\right)_{\mathbb R^d \setminus
	 \Gamma}
	+ \alpha_k \left(\kappa \nabla \widehat{e}_k^h(t), \nabla w^h\right)_{\Omega_-} + 		\alpha_k \left(\nabla \widehat{e}_
	  \kappa^\star(t), \nabla w^\star\right)_{\mathbb R^d \setminus \Gamma} \\
	 = (\chi_k^h(t),w^h)_{\Omega_-} + (\chi_k^\star(t),w^\star)_{\mathbb R^d \setminus \Gamma} 
	\quad \forall (w^h,w^\star) \in \mathcal{U}.
\end{alignat*}
We know by \eqref{eq:4.4c} and \eqref{eq:4.4d} that $(\widehat{e}^h_k(t), \widehat{e}^\star_k(t)) \in \mathcal{U}$ for all $t$.  We can then test the latter identity with $2\partial_k (\widehat{e}^h_k(t), \widehat{e}^\star_k(t)) = 2\alpha_k (\widehat{f}_k^h(t), \widehat{f}_k^\star(t))$ and re-order terms to obtain 
\begin{alignat*}{6}
\vertiii{\left( \widehat{e}_k^h(t), \widehat{e}_k^\star(t), \widehat{f}_k^h(t), \widehat{f}_k^\star(t)\right)}^2
	  = &  \vertiii{\left( \widehat{e}_k^h(t-k), \widehat{e}_k^\star(t-k), \widehat{f}_k^h(t-k), \widehat{f}_k^\star(t-k)\right)}^2\\
	 & + k \left( \chi_k^h(t), 2\alpha_k \widehat{f}_k^h(t)\right)_{\Omega_-} 
	 + k \left( \chi_k^\star(t), 2 \alpha_k \widehat{f}_k^\star(t) \right)_{\mathbb R^d \setminus \Gamma},
\end{alignat*}
where 
$$
\vertiii{\left(u,u^\star,v,v^\star\right)}^2:=\| c^{-1} v \|_{\Omega_-}^2 +\| \kappa^{1/2} \nabla u \|_{\Omega_-}^2 + \| v^\star \|_{\mathbb R^d \setminus \Gamma}^2 + \| \nabla u^\star \|_{\mathbb R^d \setminus \Gamma}^2.
$$
By induction,
\begin{alignat*}{6}
& \vertiii{\left( \widehat{e}_k^h(t), \widehat{e}_k^\star(t), \widehat{f}_k^h(t), \widehat{f}_k^\star(t)\right)}^2 \\
&  \hspace{.3in}= k \sum_{j \geq 0} \left(\left(\chi_k^h(t-t_j), 2 \alpha_k \widehat{f}_k^h(t-t_j)\right)_{\Omega_-}+ \left(\chi_k^\star(t-t_j), 2\alpha_k \widehat{f}_k^\star(t-t_j) \right)\right)_{\mathbb{R}^d \setminus\Gamma},
\end{alignat*}
where for each $t$ the sum is finite because all of the functions are causal.    We now take $t^\star \leq t$ such that 
$$
\vertiii{\left( \widehat{e}_k^h(t^\star), \widehat{e}_k^\star(t^\star), \widehat{f}_k^h(t^\star), \widehat{f}_k^\star(t^\star)\right)} = \max_{0 \leq \tau \leq t} \vertiii{\left( \widehat{e}_k^h(\tau), \widehat{e}_k^\star(\tau), \widehat{f}_k^h(\tau), \widehat{f}_k^\star(\tau)\right)}
$$
and therefore we can bound
\begin{alignat}{6}\label{eq:errorAlmostDone}
\nonumber
& \vertiii{\left( \widehat{e}_k^h(t^\star), \widehat{e}_k^\star(t^\star), \widehat{f}_k^h(t^\star), \widehat{f}_k^\star(t^\star)\right)}^2 \\ 
& \hspace{.5in} \leq 2 t^\star \vertiii{\left( \widehat{e}_k^h(t^\star), \widehat{e}_k^\star(t^\star), \widehat{f}_k^h(t^\star), \widehat{f}_k^\star(t^\star)\right)} \max_{0 \leq \tau \leq t^\star}\left( \|  \chi_k^h(\tau)\|_{\Omega_-} + \|\chi_k^\star(\tau)\|_{\mathbb R^d \setminus \Gamma} \right).
\end{alignat}
A simple Taylor expansion shows the following estimate of the consistency error for the trapezoidal rule:
$$
\| \chi_k^h (\tau) \|_{\Omega_-} + \| \chi_k^\star (\tau)\|_{\mathbb R^d \setminus \Gamma} \leq C k^2 \left( \max_{\tau -2k \leq \rho \leq \tau} \|(u^h(\rho))^{(4)} \|_{\Omega_-} + \max_{\tau -2k \leq \rho \leq \tau}  \| (u^\star(\rho))^{(4)}\|_{\mathbb R^d \setminus \Gamma} \right),
$$
which, combined with \eqref{eq:errorAlmostDone}, yields the error estimate
$$
\vertiii{\left( \widehat{e}_k^h(t), \widehat{e}_k^\star(t), \widehat{f}_k^h(t), \widehat{f}_k^\star(t)\right)} \leq C k^2 t \left(\max_{0\leq \tau\leq t} \|(u^h(\tau))^{(4)} \|_{\Omega_-} + \max_{0\leq \tau\leq t} \| (u^\star(\tau))^{(4)}\|_{\mathbb R^d \setminus \Gamma} \right).
$$
Applying the estimates from Corollary \ref{cor:3.1}, we have the final bound in the natural energy norm
$$
\vertiii{\left( \widehat{e}_k^h(t), \widehat{e}_k^\star(t), \widehat{f}_k^h(t), \widehat{f}_k^\star(t)\right)} \leq C k^2 t \left( H_3(\beta^{(3)}_0,t| H^{1/2}(\Gamma)) + H_2(\beta^{(3)}_1, t| H^{-1/2}(\Gamma))\right),
$$
where the constant $C$ is independent of $h$ and $t$.  This proves \eqref{eq:4.5}.
%%%%%%%%%%%%%%%
If we expand the differencing operator acting on the quantities $e_k^h(t)$ and $e_k^\star(t)$, we find
$$
\| e_k^h(t) \|_{\Omega_-} + \| e^\star_k(t)\|_{\mathbb R^d \setminus \Gamma} \leq \| e_k^h(t-k) \|_{\Omega_-} + \| e^\star_k(t-k)\|_{\mathbb R^d \setminus \Gamma} + k \|f_k^h(t)\|_{\Omega_-} + k \| f_k^\star(t)\|_{\mathbb R^d \setminus \Gamma}.
$$
We may then proceed as before and show the $L^2(\Omega_-) \times L^2(\mathbb R^d)$ error bound
$$
\| e_k^h(t) \|_{\Omega_-} + \| e^\star_k(t)\|_{\mathbb R^d \setminus \Gamma} \leq C k^2 t^2 \left(H_3(\beta^{(3)}_0, t| H^{1/2}(\Gamma)) + H_2(\beta^{(3)}_1, t| H^{-1/2}(\Gamma))\right),
$$
which establishes \eqref{eq:4.6}.  To prove \eqref{eq:4.7} we apply the trace theorem and the previous $L^2(\mathbb{R}^d\setminus\Gamma)$ and $H^1(\mathbb{R}^d\setminus\Gamma)$ estimates:
\begin{alignat*}{6}
\| \alpha_k^2\llbracket \gamma e_k^\star(t)\rrbracket \|_{1/2, \Gamma} 	& \leq C \left( \| \alpha_k \nabla \widehat{e}_k^\star(t) \|_{\mathbb{R}^d\setminus\Gamma} + \| \alpha_k^2 e_k^\star(t) \|_{\mathbb{R}^d\setminus\Gamma} \right) \\													& \leq  C   \left(\max_{0\leq\tau\leq t} \| \nabla \widehat{e}_k^\star(\tau) \|_{\mathbb{R}^d\setminus\Gamma} + \max_{0\leq\tau\leq t} \|e_k^\star(\tau)\|_{\mathbb{R}^d\setminus\Gamma}\right)\\
											& \leq C k^2 \max\{t,t^2\} \left(H_3(\beta_0^{(3)},t |H^{1/2}(\Gamma)) + H_2(\beta_1^{(3)},t | H^{-1/2}(\Gamma))\right).
\end{alignat*}
Note that 
\begin{alignat*}{6}
\| \alpha_k^2 \llbracket \partial_\nu e_h^\star(t) \rrbracket \|_{-1/2,\Gamma} \leq & C\left( \| \alpha_k^2 \nabla e_k^\star(t)\|_{\mathbb{R}^d \setminus\Gamma} + \| \alpha_k^2 \Delta e_k^\star(t) \|_{\mathbb{R}^d \setminus\Gamma}\right) \\
														\leq & C \left( \max_{0\leq \tau \leq t} \| \nabla \widehat{e}_k^\star(\tau)\|_{\mathbb{R}^d\setminus\Gamma} + 																										\| \partial_k \widehat{f}_k^\star(t) \|_{\mathbb{R}^d \setminus\Gamma} + \| \chi_k^\star(t)\|_{\mathbb{R}^d\setminus\Gamma} \right) \\	
														\leq & C \left( \max_{0\leq \tau \leq t} \| \nabla \widehat{e}^\star(\tau)\|_{\mathbb{R}^d\setminus\Gamma} + \max_{0\leq \tau \leq t} \left\| \frac{d}{dt}	\widehat{f}_k^\star(\tau) \right\|_{\mathbb{R}^d\setminus\Gamma} + \|\chi_k^\star(t)\|_{\mathbb{R}^d\setminus\Gamma} \right)									
\end{alignat*}
where we have applied \eqref{eq:4.4b} and the Mean Value Theorem.  The final bound \eqref{eq:4.8} follows from the previous estimates and the fact that the error corresponding to data $(\dot{\beta}_0, \dot{\beta}_1)$ is the time derivative of the error.  This is due to the fact that all operators involved are convolution operators.  This finishes the proof.
\end{proof}

\subsection{Algorithm}\label{sec:4.2}

We fix a basis for the finite dimensional space $U_h$ (the FEM space) and for the spaces $X_h$ and $Y_h$ (the BEM spaces).   Let $\mathrm V_h(s)$, $\mathrm K_h(s)$, $\mathrm W_h(s),$ and $\mathrm I_h$ be the matrix representations of the bilinear forms 
\begin{alignat*}{6}
\langle \cdot, V(s) \cdot \rangle &:X_h \times X_h \rightarrow \mathbb{C}, 	&& \quad \langle \cdot, K(s) \cdot \rangle :X_h \times Y_h \rightarrow \mathbb{C}, \\
\langle W(s) \cdot, \cdot \rangle& : Y_h \times Y_h \rightarrow \mathbb{C}, &&\quad \langle \cdot, \cdot \rangle: X_h \times Y_h \rightarrow \mathbb{R}.
\end{alignat*}
These matrix-valued functions of $s$ involve only the boundary element spaces. Let $\mathrm M_h$ and $\mathrm S_h$ be the finite element mass and stiffness matrices, that is, the matrix representation of the symmetric bilinear forms
\[
(c^{-2} \cdot,\cdot)_{\Omega_-}:U_h\times U_h \to \mathbb R,
\qquad
(\kappa \nabla\cdot,\nabla\cdot)_{\Omega_-}:U_h\times U_h \to \mathbb R.
\]
Finally, let $\Gamma_h$ be the matrix representation of $\langle \cdot , \gamma \cdot \rangle : X_h \times U_h \rightarrow \mathbb{R}.$   This is the only matrix that connects the finite and boundary element spaces, a connection simply established through inner products.   

For simplicity of exposition, let us assume that the functions $\beta_0(t_n)$ and $\beta_1(t_n)$ have been projected or interpolated onto the spaces $Y_h$ and $X_h$, respectively.   The corresponding vectors of coefficients will be denoted $\boldsymbol{\beta}_{1,n}$ and $\boldsymbol{\beta}_{0,n}$.   The marching-on-in-time scheme for discretization inverts the same large coupled operator at each time step, and then updates the right hand side with past values of the solution. It can be interpreted in the following form: in the interior domain we have a trapezoidal rule discretization of the FEM-semidiscrete wave equation with Neumann (unknown) boundary conditions.
\begin{alignat}{6}\label{eqCQFEM}
\frac{4}{k^2}\mathrm M_h \mathbf{u}_n +\mathrm S_h \mathbf{u}_n - \Gamma_h^t \boldsymbol{\lambda}_n	=	
&   \Gamma_h^t \left( \boldsymbol{\beta}_{1,n} + 2\boldsymbol{\beta}_{1,n-1} + \boldsymbol{\beta}_{1,n-2} - 2\boldsymbol{\lambda}_{n-1} + \boldsymbol{\lambda}_{n-2} \right)  \\
\nonumber
		& - \frac{1}{k^2}\mathrm M_h \left( 2 \mathbf{u}_{n-1} - \mathbf{u}_{n-2} \right) 
		+\mathrm S_h\left( 2\mathbf{u}_{n-1} - \mathbf{u}_{n-2}\right),
\end{alignat}
while in the exterior domain a trapezoidal rule CQ scheme discretizes a symmetric Galerkin-BEM system with given (yet unknown) Dirichlet data
\begin{alignat}{6}\label{eqCQBEM}
	\left[ 
	\begin{array}{c}
	\Gamma_h \mathbf{u}_n  \\
		0
	\end{array}
	\right]
	&+
	\left[
	\begin{array}{cc}
	\mathrm{V}_h(2/k)							&	-\frac{1}{2} \mathrm I_h +\mathrm{K}_h(2/k)  \\
	\frac{1}{2} \mathrm I_h^t + \mathrm{K}_h^t(2/k)		&	\mathrm{W}_h(2/k)
	\end{array}
	\right]
	\left[
	\begin{array}{c}
	\boldsymbol{\lambda}_n \\
	\boldsymbol{\phi}_n 
	\end{array}
	\right] \\
&\nonumber\hspace{2cm}= 
	\left[
	\begin{array}{c}
	\mathrm I_h \boldsymbol{\beta}_{0,n} \\
	0
	\end{array}
	\right]
	-  \sum_{m=1}^n
	\left[
	\begin{array}{cc}
	\omega_m^{\mathrm{V}_h}(k)		&	\omega_m^{\mathrm{K}_h}(k) 	\\
	\omega_m^{\mathrm{K}_h^t}(k)			&	\omega_m^{\mathrm{W}_h}(k) 
	\end{array}
	\right]
	\left[
	\begin{array}{c}
	\boldsymbol{\lambda}_{n-m} 	\\
	\boldsymbol{\phi}_{n-m}
	\end{array}
	\right].
\end{alignat}
As we progressively compute the vectors $\mathbf u_n$, $\boldsymbol\lambda_n$, and $\boldsymbol\phi_n$, we can input the latter two in the CQ-discretized potential expression (using the basis representation for elements of $Y_h$ and $X_h$):
$$
u^\star_k(t_n) = \sum_{m=0}^n \omega_m^{\mathrm{D}_h}(k)\boldsymbol{\phi}_{n-m} - \sum_{m=0}^n \omega_m^{\mathrm{S}_h}(k)\boldsymbol{\lambda}_{n-m}.
$$
The convolution weights $\omega_m^J(k)$ for $J \in \{ \mathrm{V}_h, \mathrm{K}_h, \mathrm{K}^t_h, \mathrm{W}_h, \mathrm{S}_h, \mathrm{D}_h\}$ are computable based on the Taylor expansion of the appropriate transfer function, as in \eqref{eq:4.1a}. Alternatively, the memory term in the right-hand side of \eqref{eqCQBEM} and the potential representations can be evaluated using FFT-based techniques \cite {BaSc:2012,HaSa:2014}.

\paragraph{Spaces of piecewise polynomials.} Let us now focus on the case when $\Omega_-$ is a polygon or polyhedron that has been partitioned into triangles or tetrahedra. We choose $U_h$ to be the space of continuous piecewise polynomial functions of degree at most $p \geq 1$, $Y_h$ to be the space of continuous piecewise polynomial functions of degree at most $p$ on the inherited partition of the boundary, and $X_h$ to be the space of discontinuous piecewise polynomial functions of degree at most $p-1$ on the same partition of the boundary. Note that the use if the inherited partition on the boundary is done for the sake of simplicity but is not a necessary theoretical assumption. For this choice of spaces, $Y_h$ can be identified with the trace space of $U_h$, and therefore, the matrix $\Gamma_h$ can be computed from $\mathrm I_h$ identifying degrees of freedom of $Y_h$ with the boundary degrees of freedom of $U_h$. In the two dimensional case, $X_h$ and $Y_h$ have the same dimension, and therefore all boundary matrices are square.

We can now give a simple error estimate for the case of smooth solutions of our problem.   Suppose, for instance, that $c$ and the components of the matrix-valued function $\kappa$ are $\mathcal{C}^\infty$, that $c\equiv 1$ in a neighborhood of $\Gamma$, and $\kappa \equiv \mathrm{I}$ (the identity matrix) in a neighborhood of $\Gamma$ as well.   Let the incident wave be a plane wave $u^{\mathrm{inc}}(t)(\mathbf x) = \psi(\mathbf x \cdot \mathbf d - t - t_0)$, where $\psi$ is a smooth causal function, $| \mathbf d | =1$, and $t_0$ is taken so that the support of $u^{\mathrm{inc}}$ does not intersect $\Omega_-$ at time $t=0$.   In this case the solutions of \eqref{eq:2.3} are smooth functions of space and time and the restriction of the boundary unknowns $\lambda$ and $\phi$ to the faces of $\Gamma$ are smooth. Using Corollary \ref{cor:3.2} and standard estimates for approximations by piecewise polynomials, we can prove that the semidiscrete error satisfies
$$
\| u(t) - u^h(t) \|_{1,\Omega_-} + \| \phi(t) - \phi^h(t)\|_{1/2,\Gamma} + \| \lambda(t) - \lambda^h(t) \|_{-1/2,\Gamma} = \mathcal{O}(h^p).
$$
Consider now the quantities 
$$
e^u_n := u^h(t_n)- u_k^h(t_n), \quad e_n^\lambda:= \lambda^h(t_n) - \lambda^h_k(t_n), \quad \text{and} \quad e^\phi_n:= \phi^h(t_n)-\phi_k^h(t_n). 
$$
Then, by Theorem \ref{th:4.2}, we can prove
\begin{alignat*}{6}
\left\| \smallfrac{1}{2} (e^u_n + e^u_{n-1}) \right\|_{1,\Omega_-} + & \left\| \smallfrac{1}{4}(e_n^\phi + 2 e_{n-1}^\phi + e^\phi_{n-2}) \right\|_{1/2,\Gamma} \\
													& + \left\| \smallfrac{1}{4}(e_n^\lambda + 2e^\lambda_{n-1} +e^\lambda_{n-2}) \right\|_{-1/2,\Gamma} = \mathcal{O}(k^2).
\end{alignat*}

\paragraph{Parallelizing computations.} In equations \eqref{eqCQFEM} and \eqref{eqCQBEM}, we can see that the finite element time stepping component of the solution has a short tail, i.e. it has a memory of only two time steps, while the boundary integral right hand sides have contributions from every previously computed time step.  Computing the convolutional tails for the boundary integral equations can become expensive.  To overcome this bottleneck, we use a reduction to the boundary method that decouples the solution process into three steps: solving first for an intermediate variable $\mathbf w$ (the result of solving an interior Neumann problem corresponding to the action of the incident wave), solving next for the boundary densities, and finally solving for the interior unknown.  While this seems to require more solves than the time stepping method, this strategy does not require the computation of the convolutional tail at each time and can therefore be implemented in parallel.   The all-steps-at once CQ method of \cite{BaSa:2008} is used for the parallel time stepping. Consider the Finite Element matrix $\mathrm F_h(s):= \mathrm S_h+ s^2 \mathrm M_h,$ which is the Laplace transform of the FEM-semidiscrete wave equation in the interior domain. The method consists of the following sequential steps:
\begin{enumerate}
\item Compute the intermediate variable $\mathbf w_n$ by solving the convolution 
$$
\sum_{m=0}^n \omega_m^{\mathrm F_h}(k) \mathbf w_{n-m} = \Gamma_h^t \boldsymbol{\beta}_{1,n} \qquad n=0,\dots,N
$$
in parallel across the time steps. Equivalently, use the trapezoidal rule (with zero initial values) for the differential equation $\mathrm M_h \ddot{\mathbf w}(t)+\mathrm S_h \mathbf w(t)=\Gamma_h^t \boldsymbol\beta_1(t),$ where $\boldsymbol\beta_1(t)$ is the projection onto $X_h$ of the actual transmission data.
\item Instead of time-stepping to compute the boundary unknowns $\boldsymbol{\lambda}_n$ and $\boldsymbol{\phi}_n$ by 
\begin{alignat*}{6}
\sum_{m=0}^n
\left[
\begin{array}{cc}
 	\omega_m^{\mathrm V_h}(k) + \Gamma_h \omega_m^{\mathrm{F}^{-1}_h}(k) \Gamma_h^t	
 		&		\omega_m^{\mathrm{K}_h}(k)  \\
	\omega_m^{\mathrm K_k^t}(k) 		&		\omega_m^{\mathrm W_h}(k)
\end{array}
\right]
\left[
\begin{array}{c}
 	\boldsymbol{\lambda}_{n-m} \\
	\boldsymbol{\phi}_{n-m}
\end{array}
\right]
&+
\frac{1}{2}
\left[
\begin{array}{c}
 - \mathrm I_h^t \boldsymbol{\phi}_n \\
   \mathrm I_h\boldsymbol{\lambda}_n
\end{array}
\right] \\
& \hspace{-2cm} = 
\left[
\begin{array}{c}
\boldsymbol{\beta}_{0,n} - \Gamma_h \mathbf w_n \\
			0
\end{array}
\right],
\quad n=0,\dots,N,
\end{alignat*}
we apply the all-steps-at-once strategy to approximate CQ solutions \cite{BaSa:2008}. This requires solving in parallel systems with matrix
\[
\mathrm B_h(s):=\left[\begin{array}{cc}
	\mathrm V_h(s) + \Gamma_h \mathrm F_h^{-1}(s) \Gamma_h^t 
		&-\frac12\mathrm I_h+ \mathrm K_h(s) \\
	\frac12\mathrm I_h^t+\mathrm K_h^t(s) & \mathrm W_h(s)
\end{array}\right]
\]
for a large number of complex frequencies $s$ (with non-zero real part). Note that the construction of the above matrix (if a direct method is to be used) requires the solution of one linear system related to $\mathrm F_h(s)$ for each column of $\Gamma_h^t$. 
\item Compute the interior unknown by 
$$
\sum_{m=0}^n \omega_m^{\mathrm F_h}(k) \mathbf{u}_{n-m} = \Gamma_h^t \left( \boldsymbol{\beta}_{1,n} + \boldsymbol{\lambda}_n \right), \qquad n =0,\dots,N
$$
or use an equivalent trapezoidal method for an interior problem (with the correct boundary data now that we have computed $\boldsymbol\lambda_n$), or use an all-steps-at-once to compute $\mathbf u_n$ using a parallel algorithm.
\end{enumerate}
The exterior solution can be postprocessed at the end of the second step. If we use an iterative method for the solution of a system associated to the matrix $\mathrm B_h(s)$, every matrix-vector multiplication requires the solution of a sparse linear system associated to the interior domain. Efficient methods to handle this discrete scheme are the goal of further investigation. (In all the numerical experiments below, system solves are handled with Matlab's backslash operator.)

\section{Numerical experiments and simulations}\label{sec:5}

We perform some numerical experiments to demonstrate the coupling scheme and corroborate our theoretical results.  The first numerical experiment is created by studying an artificial scattering problem on the domain $[-0.5,0.5]^2$. Instead of an incident wave, we generate transmission data on $\Gamma$ so that the solution in the interior and exterior domains is known exactly.  In the interior, we take the solution to be a plane wave moving in the direction $(1/\sqrt{2}, 1/\sqrt{2})$ and transmitting the signal $\sin(2t)\chi(t)$ where $\chi(t)$ is a smooth cutoff function so the signal has compact support in time.  The exterior solution is a cylindrical wave due to a source point at the origin transmitting the signal $\sin^6(4t)H(t)$ where $H(t)$ is the Heaviside function.   We take $c \equiv 1$ and 
$$
\kappa(x,y) = 
\left(
\begin{array}{cc}
1 + 0.5\left(x^2 + y^2\right) 	&	0.25 + 0.5 \left(x^2 +y^2\right) 	\\
0.25 + 0.5 \left(x^2 +y^2\right)	&	3 + 0.5 \left(x^2 +y^2\right)	\\
\end{array}
\right).
$$
A body force term $f(x,y,t)$ is added in the interior domain (the equation is thus $c^{-2}\ddot u=\mathrm{div}(\kappa\nabla u)+f$) so that the chosen function (a plane wave) satisfies the wave equation in $\Omega_-$.  We discretize in space with standard $\mathbb P_1$ FEM for the interior variable and $\mathbb P_1 \times \mathbb P_0 $ BEM (i.e., $Y_h$ and $X_h$ are respectively spaces of continuous $\mathbb P_1$ and discontinuous $\mathbb P_0$ functions) for the boundary unknowns.  The simulation is run from $t=0$ to $t=3$ so that by the final time the exact solution is non-zero in both sides of the transmission boundary.  Time discretization is carried out with trapezoidal rule based CQ.  

For our error quantities, we use the following measures:
\begin{alignat*}{6}
&\mathrm{E}^u_{L^2}(t) := \| u(t) - u_k^h(t) \|_{\Omega_-} ,					&& \mathrm{E}^u_{H^1}(t):= \| u(t) - u_k^h(t)\|_{1,\Omega_-}, \\
&\mathrm{E}^\lambda(t) := \| \lambda(t) - \lambda_k^h(t) \|_{\Gamma},		&& \mathrm{E}^\phi(t) := \| \phi(t) - \phi_k^h(t)\|_\Gamma, \\
& \mathrm{E}^{\mathrm{obs}}(t) : = \max_j |u_+(t)(\mathbf x_j) - u_k^\ast(t)(\mathbf x_j) |.
\end{alignat*}
In $\mathrm{E}^{\mathrm{obs}}(t)$,  $\{\mathbf x_j\} $ is a finite collection of points in $\Omega_+$.   Note that we do not have any result asserting superconvergence in the $L^2(\Omega_-)$ norm for $u$, superconvergence in the $L^2(\Gamma)$ norm for $\phi$, or convergence in the $L^2(\Gamma)$ norm for $\lambda$.

Tables \ref{table:1} and \ref{table:2} summarize the results from the convergence study.   We use uniform triangulations with $N_{FEM}$ elements in $\Omega_-$ and $N_{BEM}$ elements on the boundary and perform $M$ time steps to reach the final time $t=3$.
\begin{table}[H]
\begin{center}
\vspace{.1in}
\begin{tabular}{|c|c|c|c|c|c|}
 $N_{FEM}$	&	$M$	&  $\mathrm{E}_{L^2}^u(3)$ 	& 	e.c.r.		&	 $\mathrm{E}_{H^1}^u(3)$ 	&	e.c.r.		\\
 \hline
32  		& 	20	&	 2.4029e-02   	&	    -		&		2.6050e-01   	&	    -		\\ 
\hline	
128 		&  	40	&	5.6609e-03	&	2.0857	&		1.2017e-01	&	1.1162	\\
\hline   
512		&  	80	&	1.4013e-03	&	2.0143	&	   	6.0233e-02   	&	0.99642	\\
\hline
2048		&  	160	&	3.4927e-04   	&	2.0043	&		3.0235e-02   	&	0.99432	\\
\hline	
8192   	&  	320	&	8.7041e-05   	&	2.0046	&		1.5147e-02   	&	0.99721	\\
\hline
32678 	&  	640	&	2.2092e-05   	&	1.9782	&		7.5796e-03   	&	0.99884	\\ 
\hline
\end{tabular}
\end{center}
\caption{Convergence of the interior variable with $\mathbb{P}_1$ FEM (coupled with $\mathbb P_1\times\mathbb P_0$ BEM) and trapezoidal rule time stepping.}\label{table:1}
\end{table}
\begin{table}[H]
\begin{center}
\begin{tabular}{|c|c|c|c|c|c|c|c|}
$N_{BEM}$ &   $M$ 	&	$\mathrm{E}^\lambda(3)$	&	e.c.r.		&	$\mathrm{E}^\phi(3)$ 	&	e.c.r		& $\mathrm{E}^{\mathrm{obs}}(3)$  	&       e.c.r 	\\
\hline
16  	&	 20	&	4.7204e-01   					&	-		&		9.1250e-02   			&	-		&	2.7533e-02		&	  -		\\
\hline
32	&	40	&	1.3196e-01   					&	1.8388	&		2.4295e-02   			&	1.9092	&	2.0929e-02		&	0.39563	\\
\hline
64  	& 	80	&	4.9760e-02   					&	1.4071	&		6.0872e-03   			&	1.9968	&	2.8444e-03		&	2.8793	\\
\hline
128  	&	160	&	1.8880e-02   					&	1.3981	&		1.5196e-03   			&	2.0021	&	6.2183e-04		&	2.1935	\\
\hline
256 	&	 320	&	7.2700e-03   					&	1.3768	&		3.8422e-04   			&	1.9837	&	1.5322e-04		&	2.0210	\\
\hline
512 	& 	640	&	3.0133e-03   					&	1.2706	&		1.0697e-04   			&	1.8448	&	3.8211e-05		&	2.0035	\\
\hline
\end{tabular}
\end{center}
\caption{Convergence of boundary and exterior variables with $\mathbb{P}_1 \times \mathbb{P}_0$ BEM (coupled with $\mathbb P_1$ FEM) and trapezoidal rule based CQ. Note that we are measuring errors for $\lambda$ in a stronger norm than the one used in the theory.}\label{table:2}
\end{table}

\paragraph{A second trapezoidal rule experiment.}  We repeat the previous experiment with the same, replacing the spatial discretization by $\mathbb{P}_2$ FEM coupled with $\mathbb{P}_2 \times \mathbb{P}_1$ BEM.   Our theory predicts order two convergence in all variables for this experiment, which was not seen in the previous example because of the use of lower order FEM and BEM.  We see comparable errors to the first experiment with reduced discretization parameters.

Tables \ref{table:1a} and \ref{table:2a} summarize the results from this convergence study.   We again use uniform triangulations with $N_{FEM}$ elements in $\Omega_-$ and $N_{BEM}$ elements on the boundary and perform $M$ time steps to reach the final time $t=3$.
\begin{table}[H]
\begin{center}
\vspace{.1in}
\begin{tabular}{|c|c|c|c|c|c|}
 $N_{FEM}$	&	$M$	&  $\mathrm{E}_{L^2}^u(3)$ 	& 	e.c.r.		&	 $\mathrm{E}_{H^1}^u(3)$ 	&	e.c.r.		\\
 \hline
8 		& 	10	&	 7.9596e-02   	&	    -		&		3.9938e-01   	&	    -		\\ 
\hline	
32		&  	20	&	1.6335e-02	&	2.2847	&		1.6148e-01	&	1.3064	\\
\hline   
128		&  	40	&	3.7889e-03	&	2.1081	&	   	2.4973e-02   	&	2.6929	\\
\hline
512		&  	80	&	8.5690e-04   	&	2.1446	&		5.1730e-03   	&	2.2713	\\
\hline	
2048   	&  	160	&	2.0934e-04   	&	2.0333	&		1.2740e-03   	&	2.0217	\\
\hline
8192 	&  	320	&	5.2069e-05  	&	2.0074	&		3.3478e-04   	&	1.9281	\\ 
\hline
\end{tabular}
\end{center}
\caption{Convergence of the interior variable with $\mathbb{P}_2$ FEM and trapezoidal rule time stepping.}\label{table:1a}
\end{table}
\begin{table}[H]
\begin{center}
\begin{tabular}{|c|c|c|c|c|c|c|c|}
$N_{BEM}$ &   $M$ 	&	$\mathrm{E}^\lambda(3)$	&	e.c.r.		&	$\mathrm{E}^\phi(3)$ 	&	e.c.r		& $\mathrm{E}^{\mathrm{obs}}(3)$  	&       e.c.r 	\\
\hline
8 	&	 10	&	4.0011e+00   					&	-		&		2.7841e-01   			&	-		&	2.1634e-02		&	  -		\\
\hline
16	&	20	&	6.6196e-01   					&	2.5956	&		4.9454e-02   			&	2.4931	&	2.3736e-02		&	-0.1338	\\
\hline
32  	& 	40	&	5.8355e-02  					&	3.5038	&		1.0361e-02   			&	2.2549	&	8.1811e-03		&	1.5367	\\
\hline
64  	&	80	&	1.3106e-02 					&	2.1546	&		2.5240e-03  			&	2.0374	&	4.7098e-04		&	4.1186	\\
\hline
128 	&     160	&	3.4291e-03  					&	1.9343	&		6.1230e-04   			&	2.0434	&	9.1814e-05		&	2.3589	\\
\hline
256 	&     320	&	1.4502e-03   					&	1.2416	&		1.5236e-04   			&	2.0068	&	2.3948e-05		&	1.9388	\\
\hline
\end{tabular}
\end{center}
\caption{Convergence of boundary and exterior variables with $\mathbb{P}_2 \times \mathbb{P}_1$ BEM and trapezoidal rule based CQ.}\label{table:2a}
\end{table}

\paragraph{Runge-Kutta based CQ.} A second numerical experiment makes use of the Runge-Kutta CQ method of \cite{BaLuMe:2011,BaMeSc:2012,LuOs:1993}.  The idea of RKCQ is similar to that of the scalar case, but rather than using a linear multistep method for the CQ discretization of the transfer function, an implicit Runge-Kutta method is instead applied.  The cost, however, is in the need to evaluate the linear systems resulting from spatial discretization each stage of the RK method.  The analysis of RKCQ methods was carried out in \cite{BaLu:2011, BaLuMe:2011} using abstract arguments in the Laplace domain: in principle, we expect the convergence order to be limited to the stage order, although potential postprocessings enjoy the full classical order of the RK method. (We also note that RKCQ methods have been reported to enjoy better dispersion properties than multistep-CQ schemes \cite{BaSc:2012}.)  The experiment below is set on the same example and triangulations as the first experiment, but is discretized in space with $\mathbb{P}_2$ FEM and $\mathbb{P}_2 \times \mathbb{P}_1$ BEM and in time with CQ based on the two-stage Radau IIa scheme, a method of classical order 3 and stage order 2.  Tables \ref{table:3} and \ref{table:4} below demonstrates convergence order more than three, which was otherwise impossible using CQ based on a linear multistep method. 
\begin{table}[H]
\begin{center}
\begin{tabular}{|c|c|c|c|c|c|}
 $N_{FEM}$	&	$M$	&  $\mathrm{E}_{L^2}^u(3)$ 	& 	e.c.r.		&	  $\mathrm{E}_{H^1}^u(3)$ 	&	e.c.r.		\\
\hline
    8 	&  	20	&	7.2998e-02	&	-		&		5.8872e-01	&	-	\\
\hline   
   32	&  	40	&	2.8039e-02	&	1.3804	&	   	2.8675e-01   	&	1.0378	\\
\hline
128	&  	80	&	5.5717e-03   	&	2.3313	&		1.1027e-01  	&	1.3787	\\
\hline	
512   &  160	&	6.8020e-04   	&	3.0341	&		3.1564e-02   	&	1.8047	\\
\hline
1024 &  320	& 	7.9143e-05	&	3.1034	&		8.3212e-03	&	1.9234	\\
\hline
2048 &  640 	&	9.6606e-06	&	3.0343	&		2.1209e-03	&	1.9721	\\
\hline
\end{tabular}
\end{center}
\caption{Convergence of the interior variable when using $\mathbb{P}_2$ FEM and two-stage Radau IIa time stepping.}\label{table:3}
\end{table}
\begin{table}[H]
\begin{center}
\begin{tabular}{|c|c|c|c|c|c|c|c|}
$N_{BEM}$ 	&      $M$ 	&	$\mathrm{E}^\lambda(3)$	&	e.c.r.		&	$\mathrm{E}^\phi(3)$	&	e.c.r		& $\mathrm{E}^{\mathrm{obs}}(3)$  	&       e.c.r 	\\
\hline
4 	&   20	&	9.2094e-01   						&	-		&		2.4561e-01   			&	-		&	6.0797e-02		&	-		\\
\hline
8  	&  40		&	4.0652e-01   						&	1.1798	&		6.5693e-02   			&	1.9026	&	2.8148e-02		&	1.1109	\\
\hline
16  	&  80		&	1.3784e-01  						&	1.5603	&		9.8444e-03  			&	2.7384	&	2.7418e-03		&	3.3598	\\
\hline
32 	& 160	&	3.4386e-02   						&	2.0031	&		9.3028e-04  			&	3.4036	&	2.1949e-04		&	3.6429	\\ 
\hline
64	  &  320	&	7.3343e-03						&	2.2291	&		6.2996e-05			&	3.8843	&	1.4604e-05		&	3.9097	\\
\hline
128	 &    640	&	1.1850e-03						&	2.6297	&		3.8242e-06			&	4.0420	&	9.9590e-07		&	3.8743	\\
\hline
\end{tabular}
\end{center}
\caption{Convergence of boundary and exterior variables when using $\mathbb{P}_2 \times \mathbb{P}_1$ BEM and two-stage Radau IIa based RKCQ. }\label{table:4}
\end{table}
\paragraph{Simulation of scattering by a single obstacle.} Next, we perform a simulation of a scattering problem with a known incident plane wave and unknown exact solution.  An incident plane wave traveling in the direction $(1/\sqrt{2},1/\sqrt{2})$ interacts with the obstacle $\Omega_- = [-0.5, 0.5]^2$.  The material properties in $\Omega_-$ have a Gaussian lensing effect described by a non-homogeneous multiple of the identity tensor $\kappa(x,y) = (1 - 1.65e^{-1/(1-r^2)})\mathrm I,$ where $r=\sqrt{x^2 +y^2},$ and we take $c\equiv 1$.
The spatial discretization makes use of $\mathbb{P}_3$ finite elements in the interior with 8192 interior elements and $\mathbb{P}_3 \times \mathbb{P}_2$ boundary elements with 256 boundary elements on $\Gamma$.  The CQ time step is $k = 4.375\times 10^{-3}$ and we integrate from $t=0$ to $t=3.5$.  Some snapshots of the scattering process are shown in Figure \ref{fig:2}.

\begin{figure}[htb]
\begin{center}
\includegraphics[scale=.3]{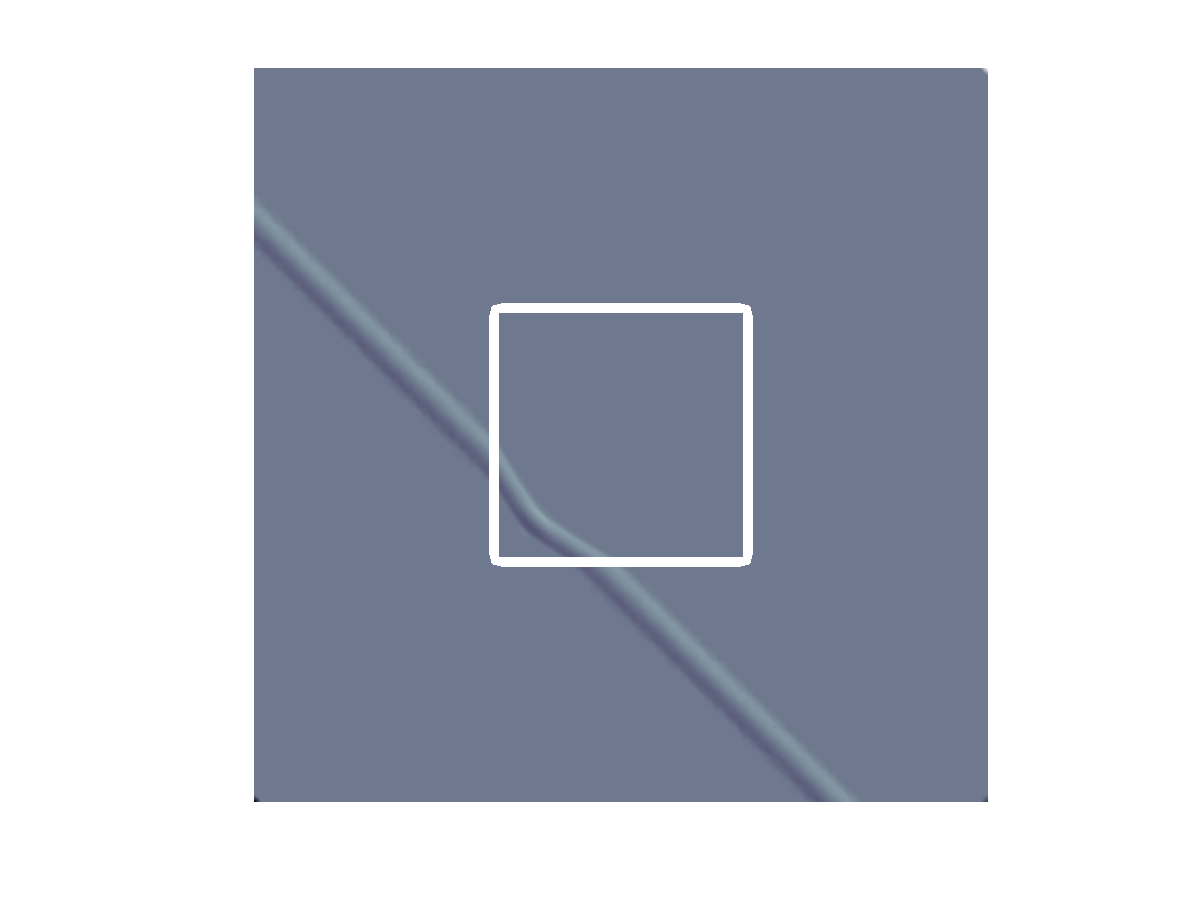}
\includegraphics[scale=.3]{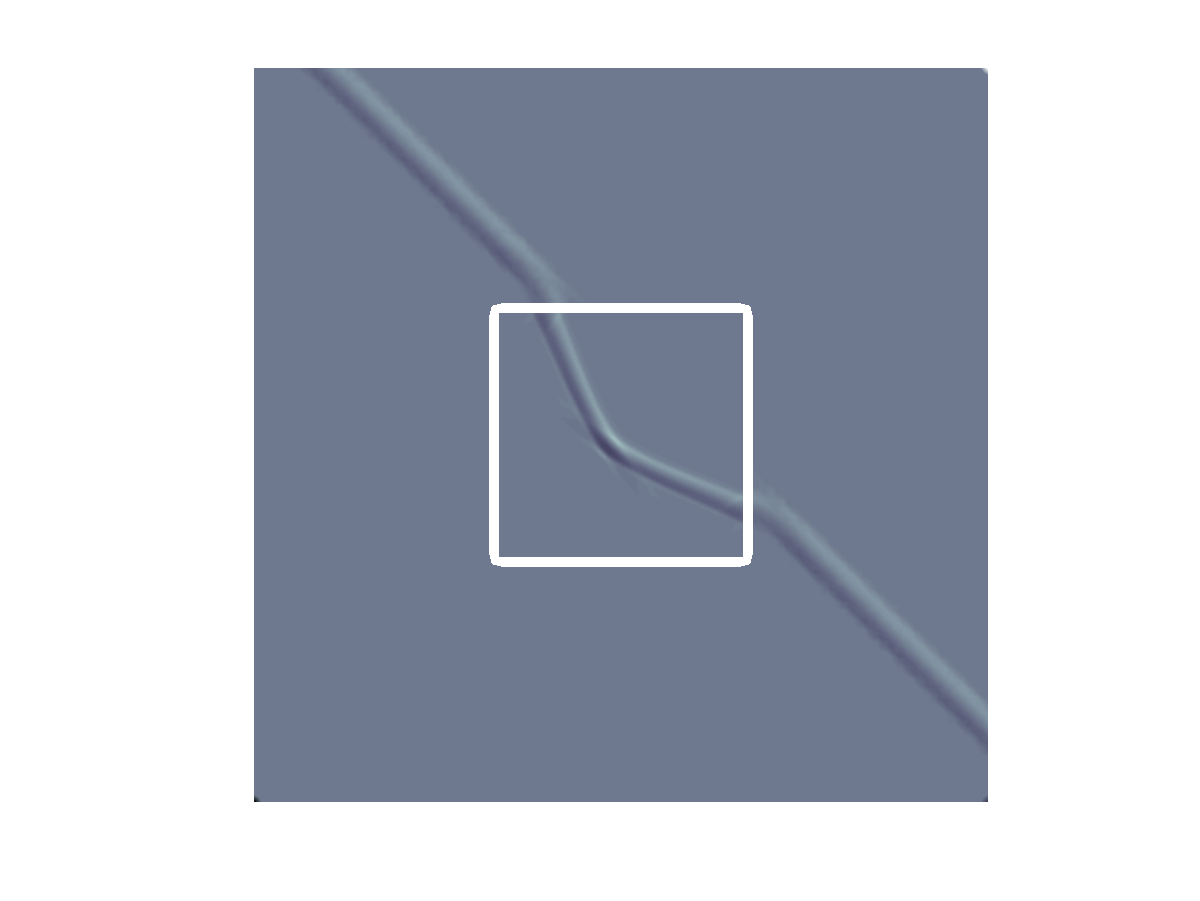}

\includegraphics[scale=.3]{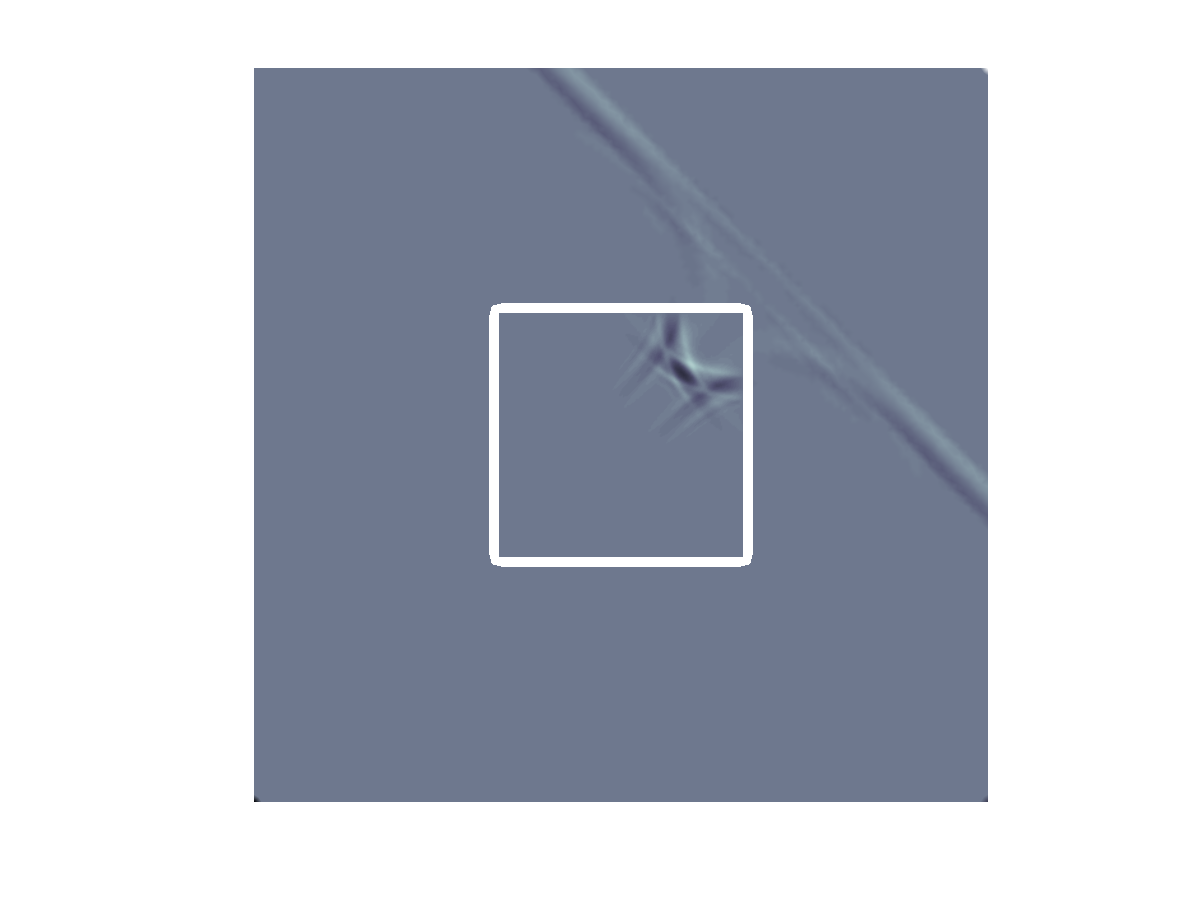}
\includegraphics[scale=.3]{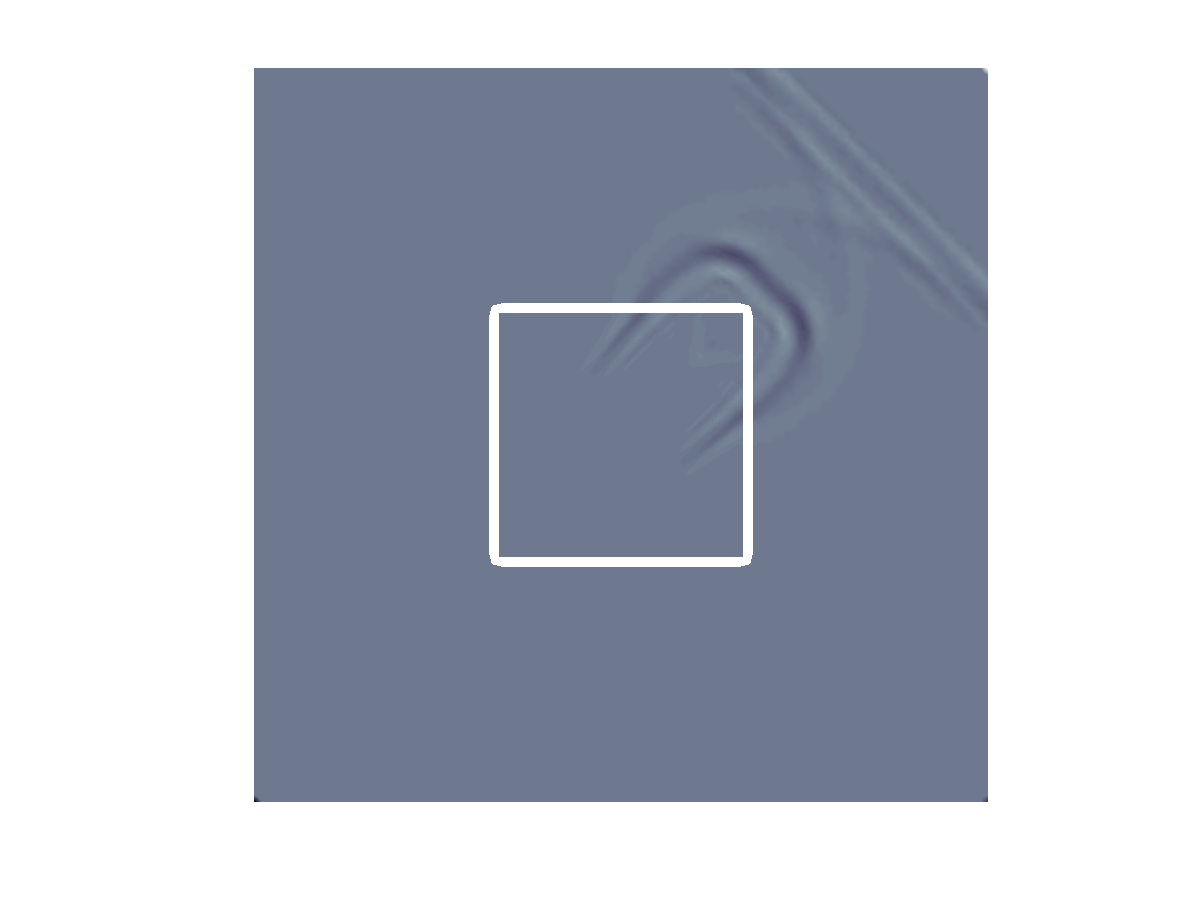}

\includegraphics[scale=.3]{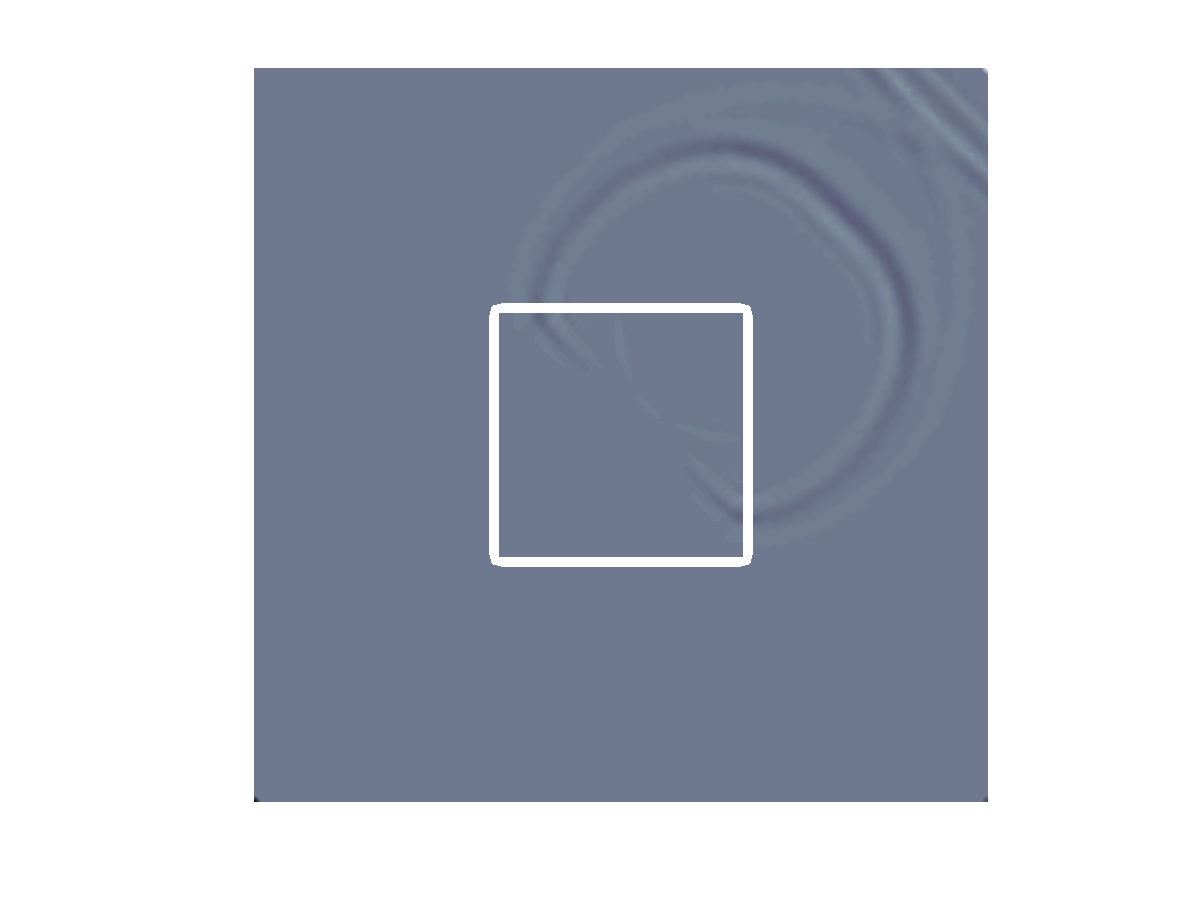}
\includegraphics[scale=.3]{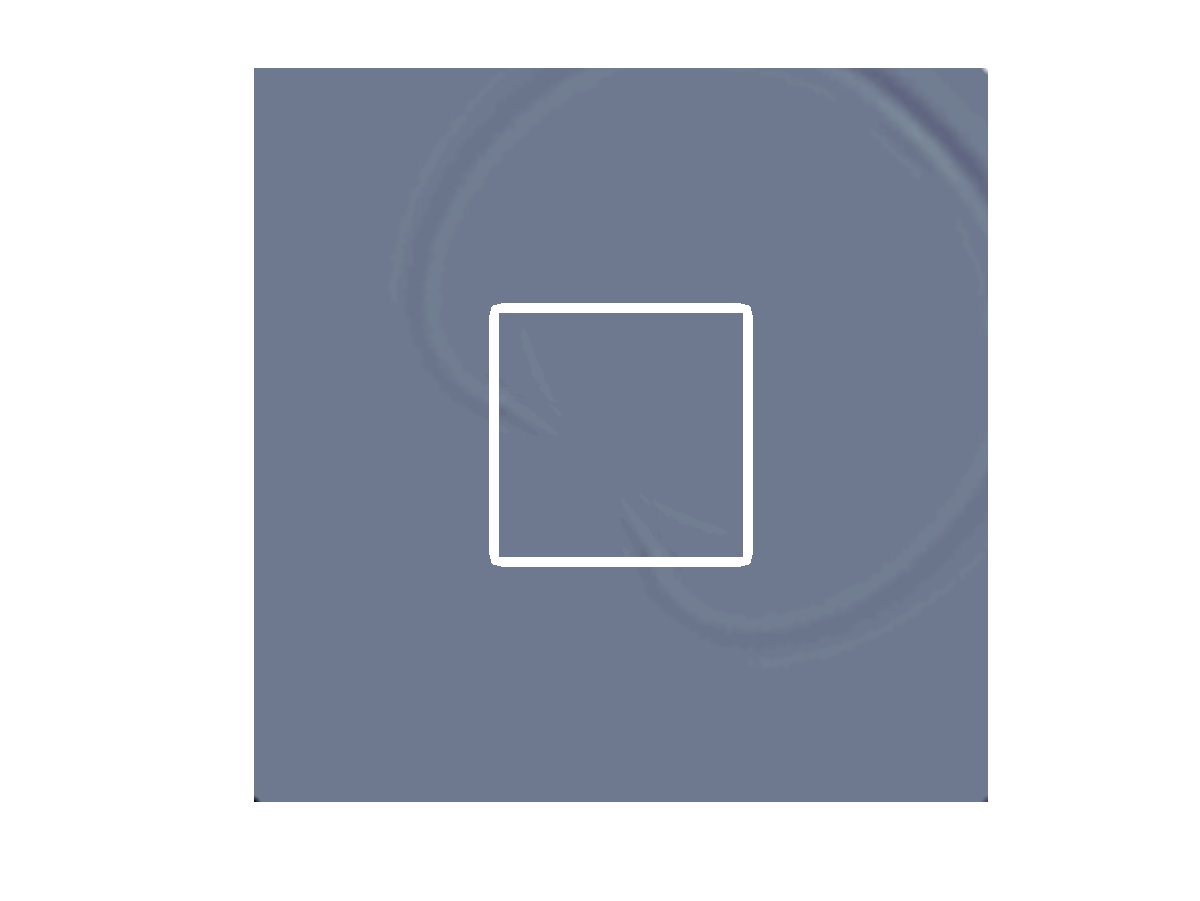}
\end{center}
\caption{Scattering of a plane wave by an obstacle with Gaussian lensing properties.}\label{fig:2}
\end{figure}

\paragraph{Simulation of scattering by multiple obstacles.} This experiment demonstrates the coupling scheme applied to multiple obstacles with different material properties.  An incident plane wave interacts with the four small boxes.  The top left and bottom right boxes have material properties described by the matrix $\kappa = \text{diag}(4, 1/4)$ while the top right and bottom left boxes have material matrix $\kappa = \text{diag}(2,1/2)$.  In all four obstacles $c\equiv 1$. Again we use $\mathbb{P}_3$ FEM and $\mathbb{P}_3\times\mathbb{P}_2$ BEM.  There are a total of 1792 finite elements and 192 boundary elements for the spatial discretization.  The time step is $k = 2 \times 10^{-2}$ and the simulation is run from $t=0$ to $t=4$. Figure \ref{fig:1} displays some different times of the experiment.

\begin{figure}[htb]
\begin{center}

\includegraphics[scale=.3]{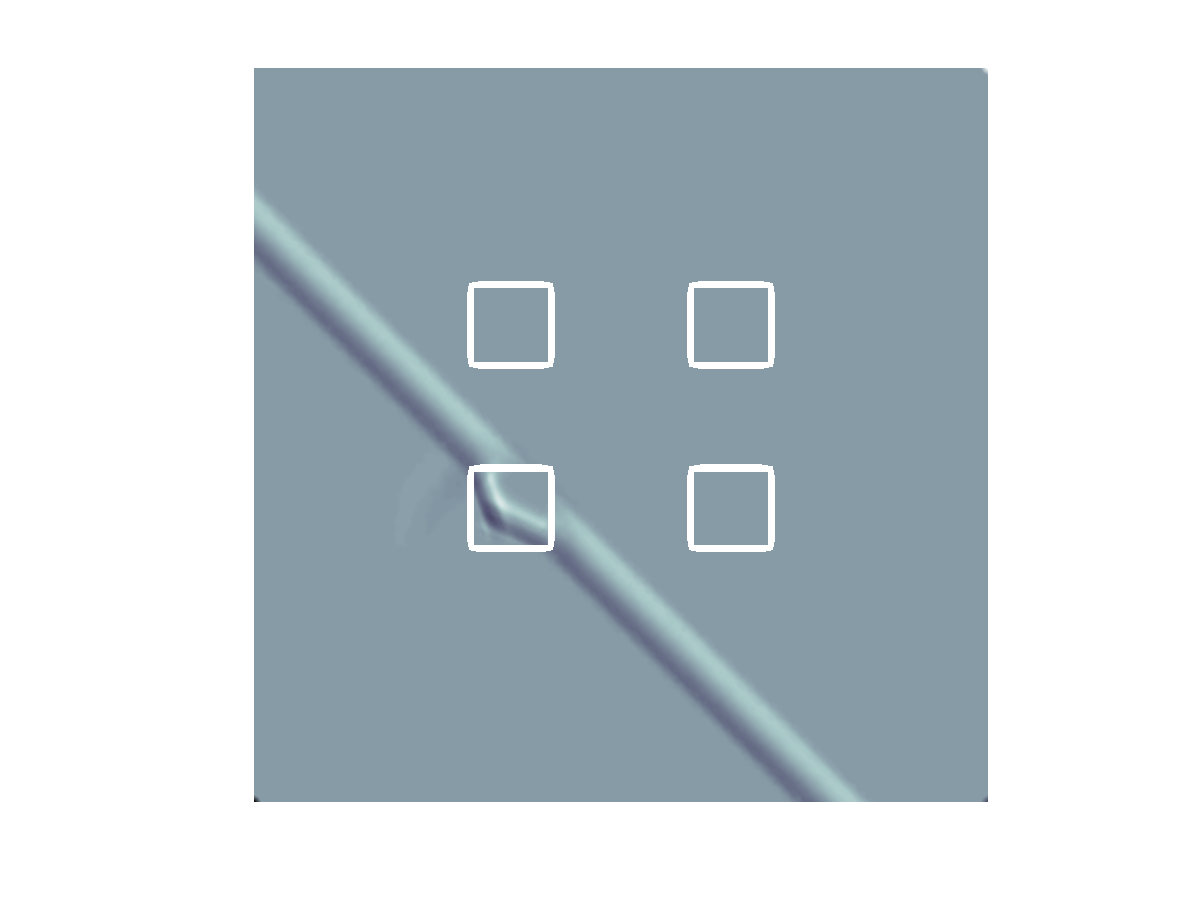}
\includegraphics[scale=.3]{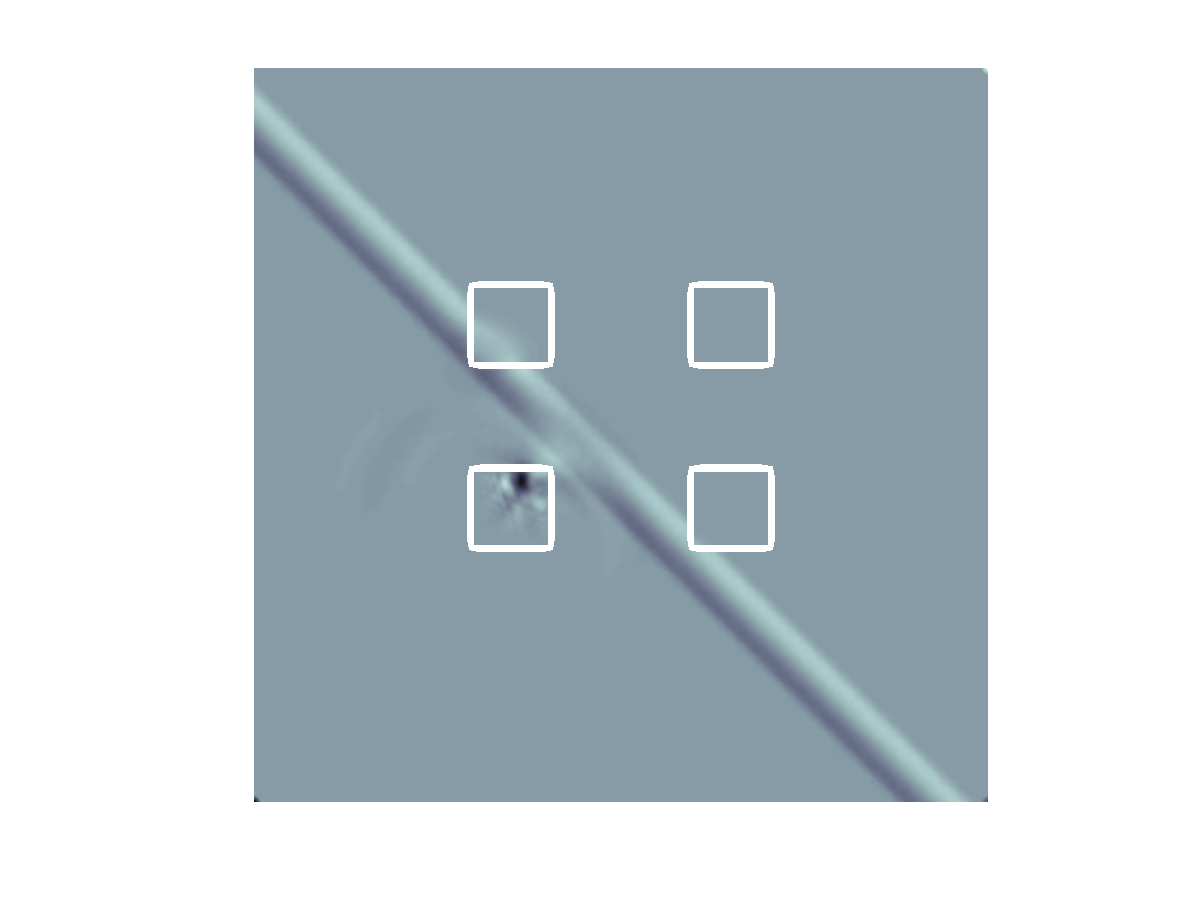}

\includegraphics[scale=.3]{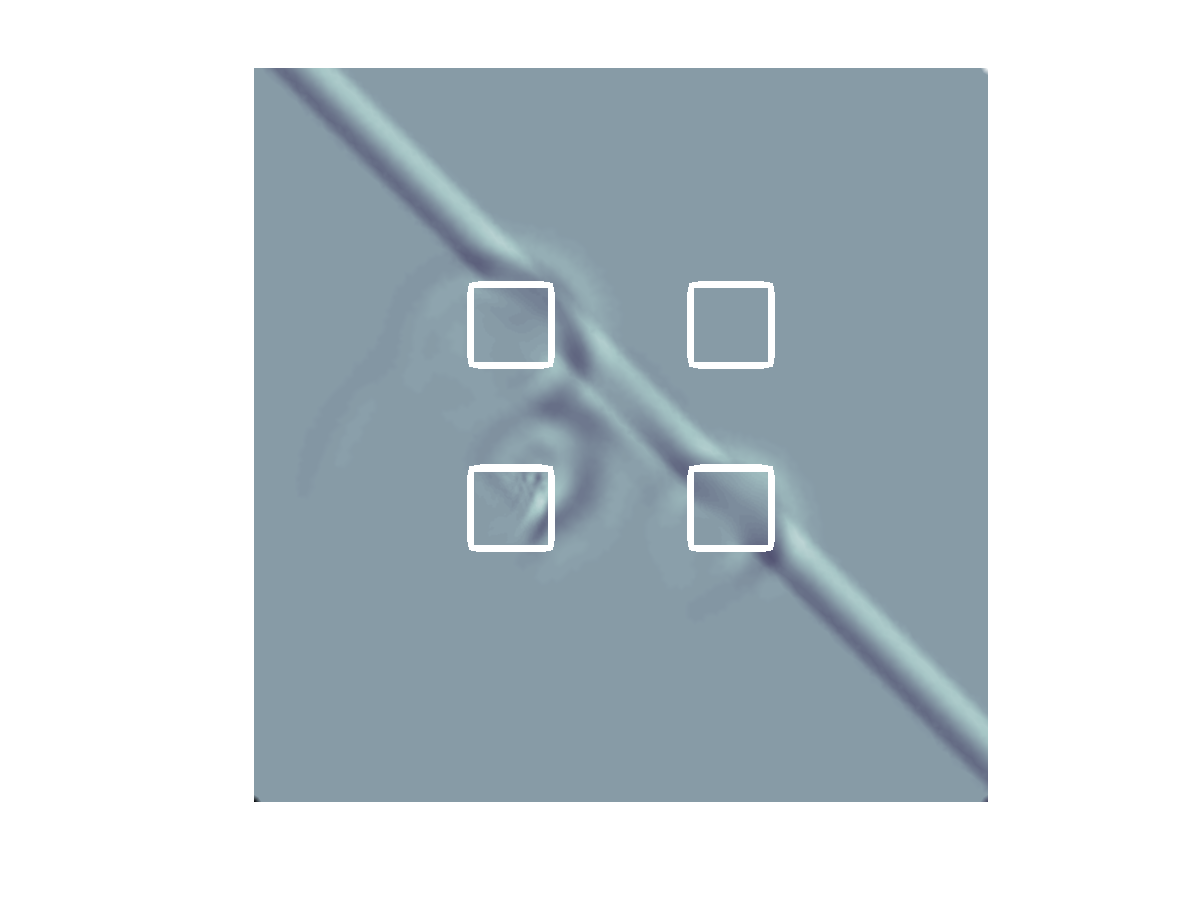}
\includegraphics[scale=.3]{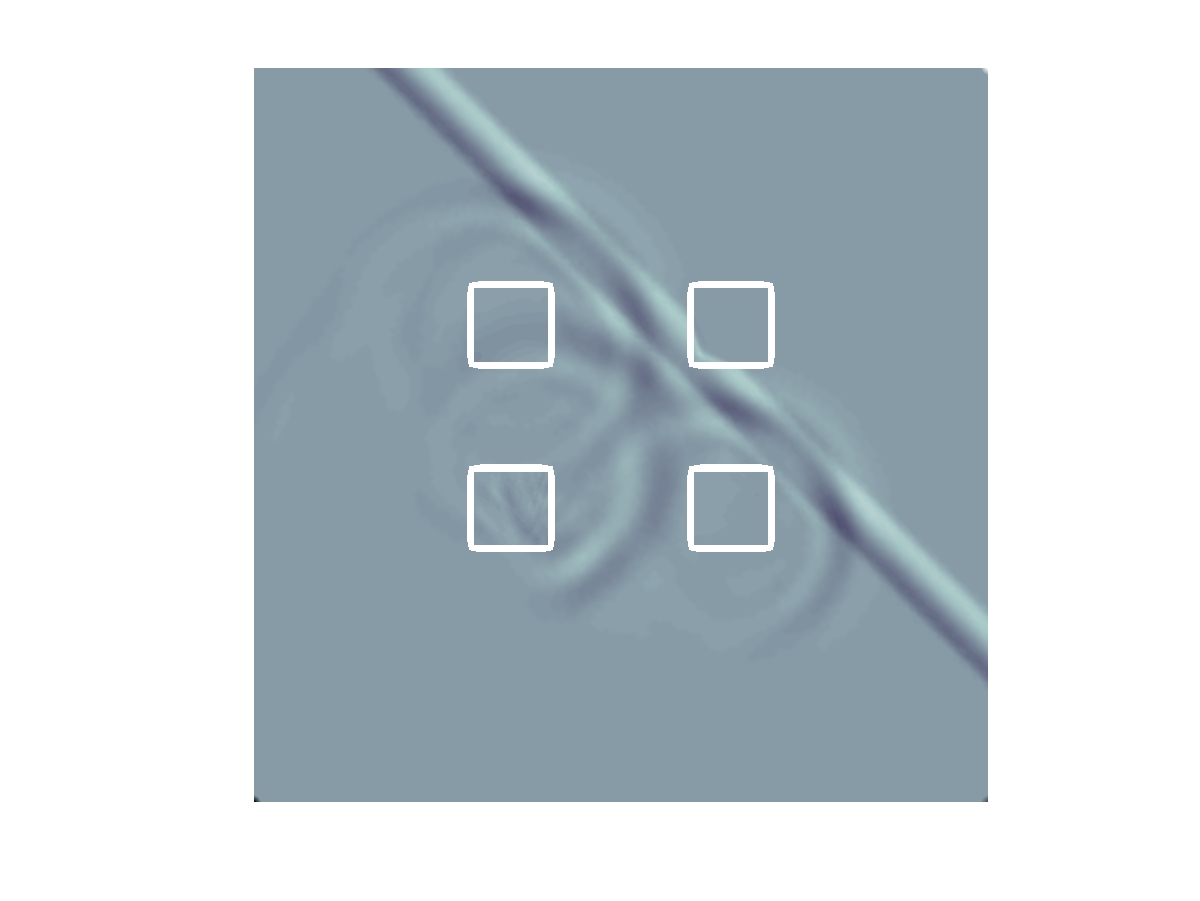}

\includegraphics[scale=.3]{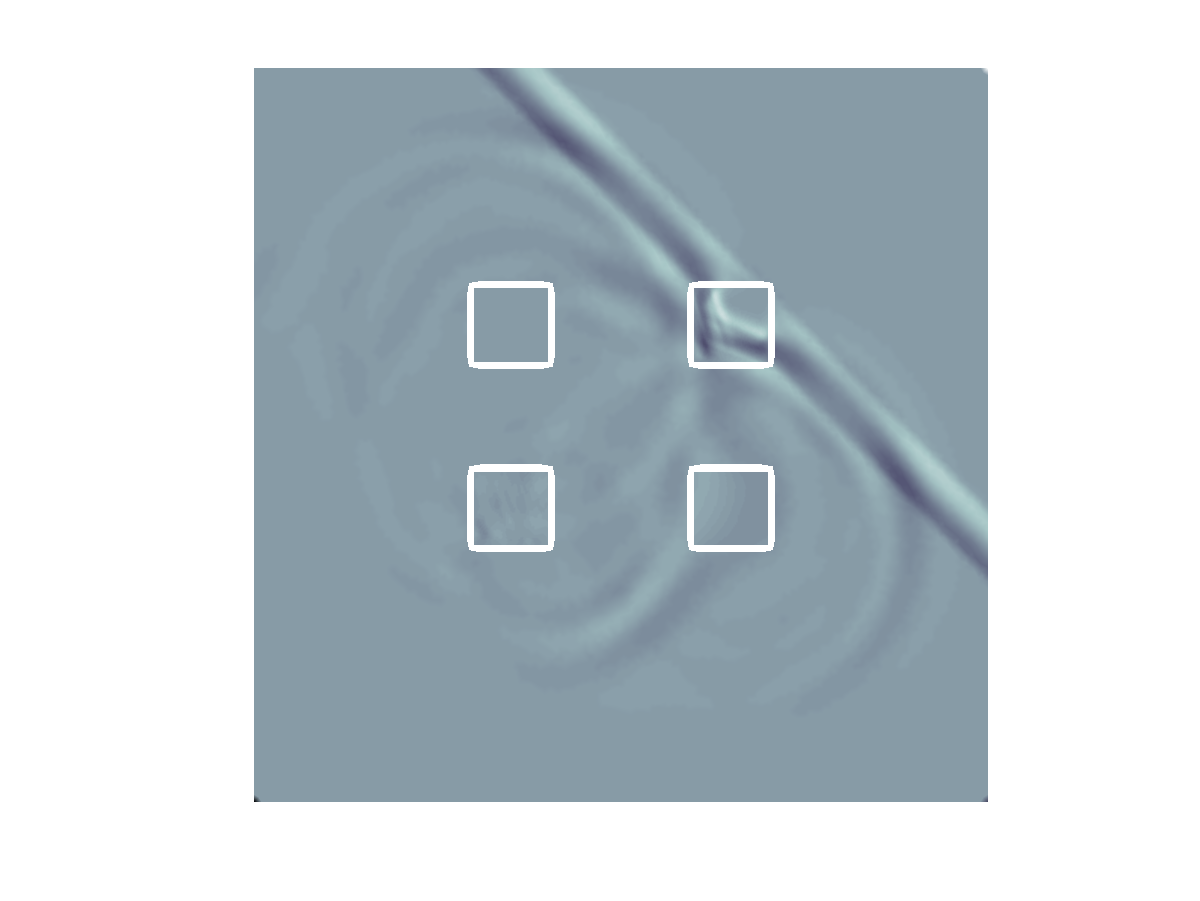}
\includegraphics[scale=.3]{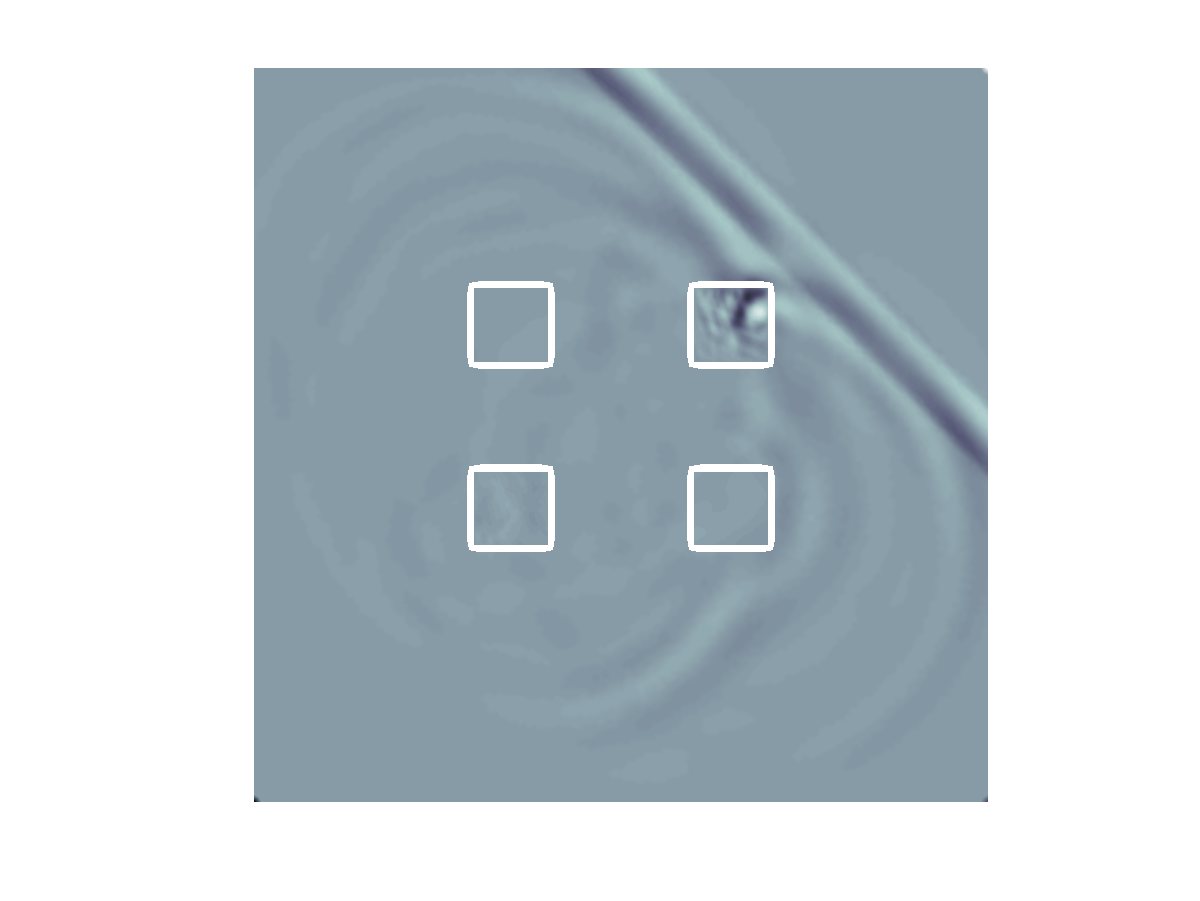}

\includegraphics[scale=.3]{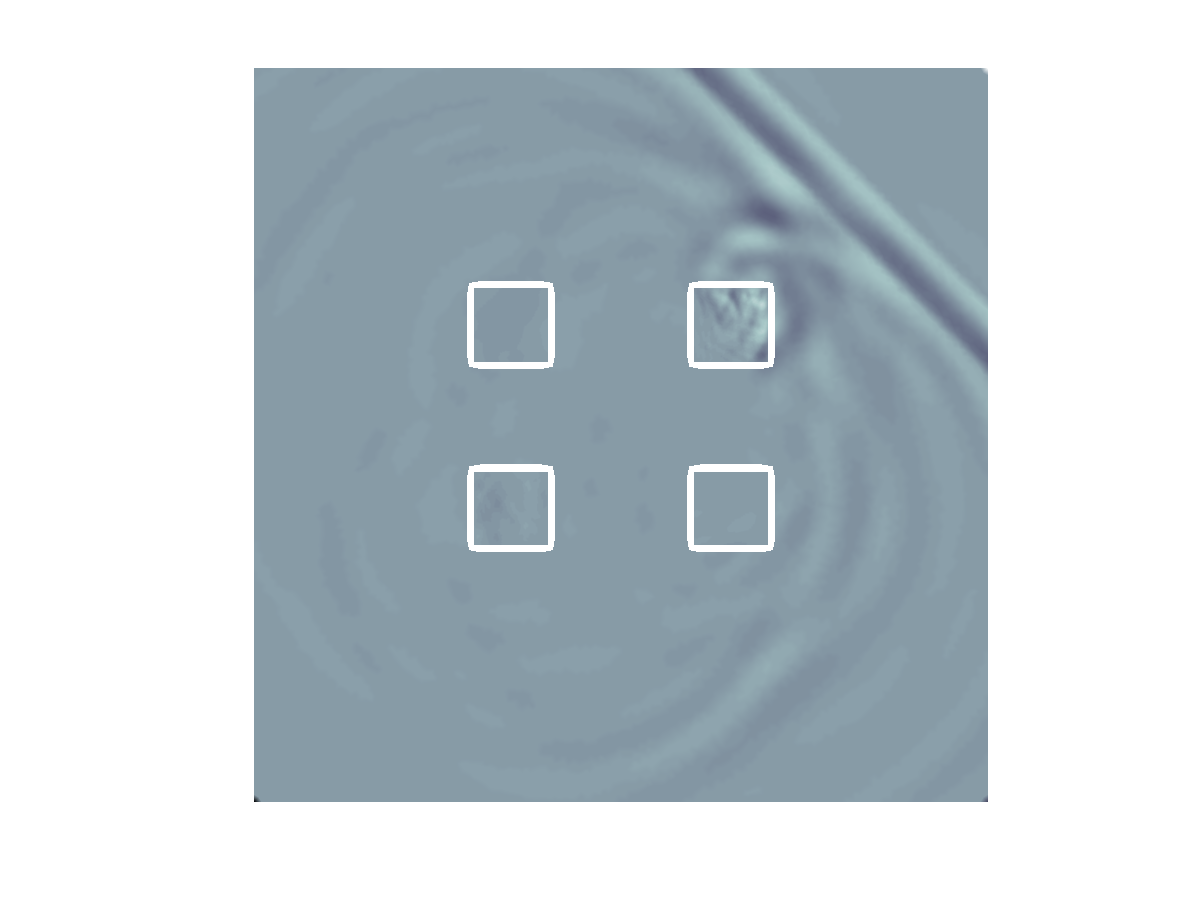}
\includegraphics[scale=.3]{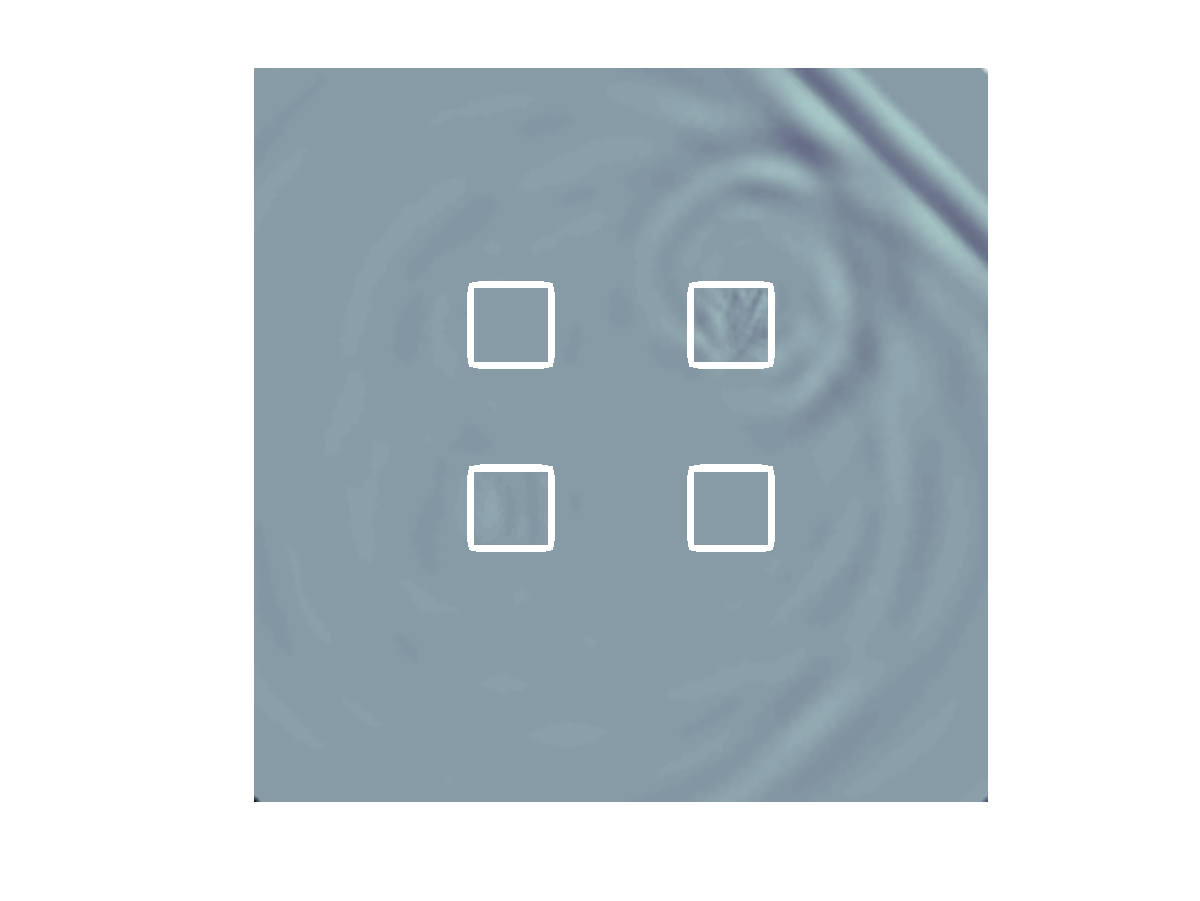}

\caption{Scattering of a plane wave by four homogeneous anisotropic obstacles with different material parameters.}\label{fig:1}
\end{center}
\end{figure}

\paragraph{Simulation of scattering by a trapping obstacle.}  Our last simulation takes place on a non-convex and trapping obstacle.   Again we use $\mathbb{P}_3$ FEM for the interior and $\mathbb{P}_3 \times \mathbb{P}_2$ BEM for the boundary densities.  The interior of the obstacle is partitioned into 11,968 finite elements, and the boundary is partitioned into 472 elements.  The time step size is $k \approx 6.7 \times 10^{-3},$ and we integrate from $t=0$ to $t=2.5$.  Wave propagation within the obstacle is determined by the parameters $c\equiv 1$ and the diagonal matrix $\kappa=\mathrm{diag}(0.25, 0.125)$.
The large difference in wave speeds between the interior and exterior produces a strong scattered wave and a highly focused and long-lived wave within the obstacle.  Some of the scattered wave is trapped within the void outside of the domain $\Omega_-$. The results are shown in Figure \ref{fig:3}.

\begin{figure}[htb]
\begin{center}
\includegraphics[scale=.3]{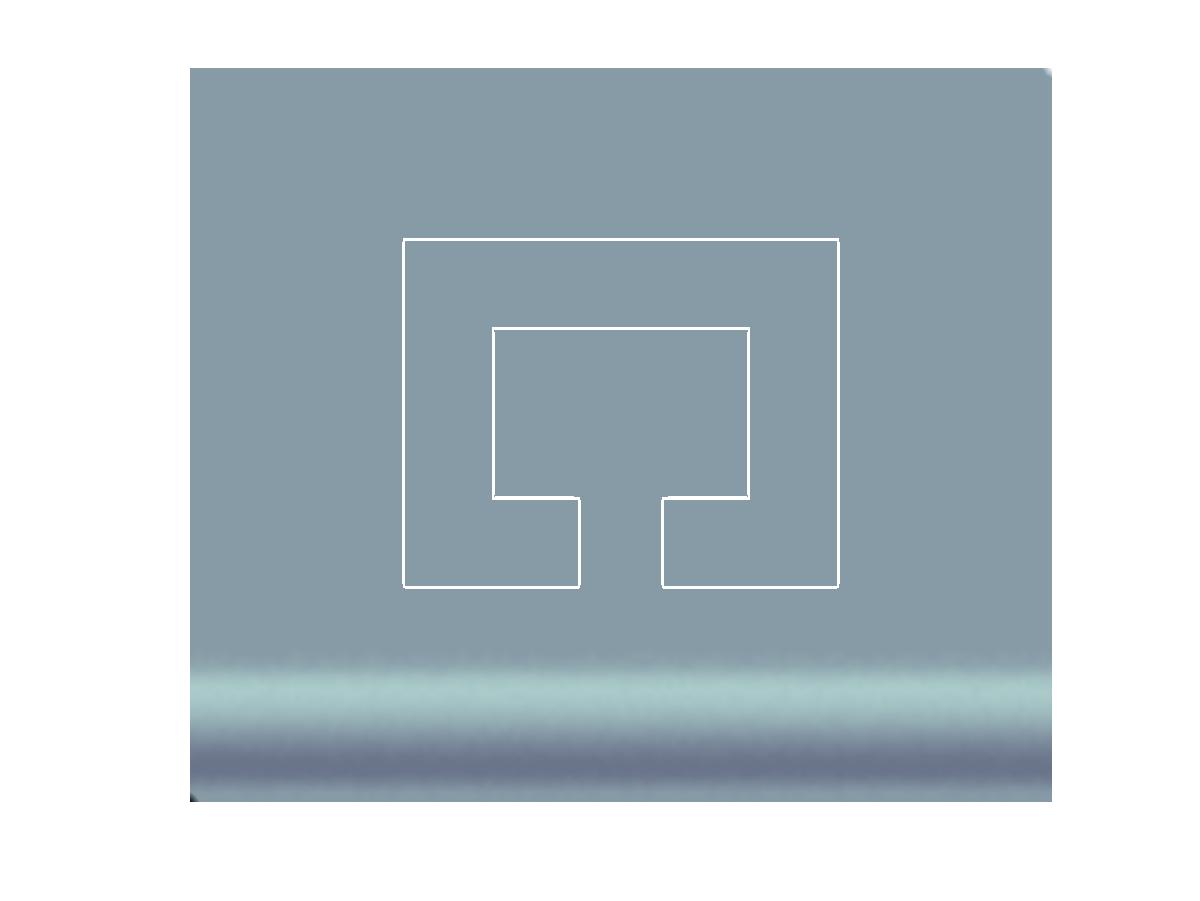}
\includegraphics[scale=.3]{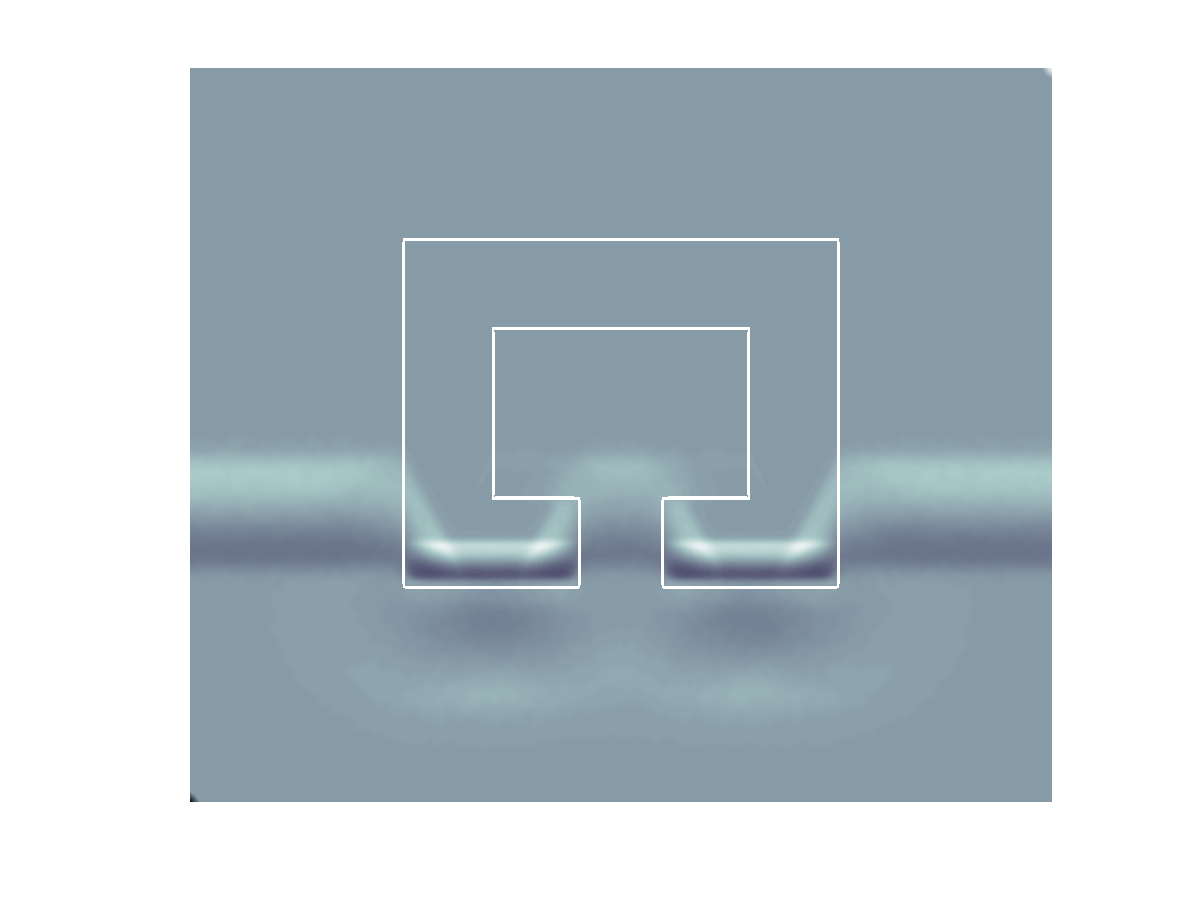}

\includegraphics[scale=.3]{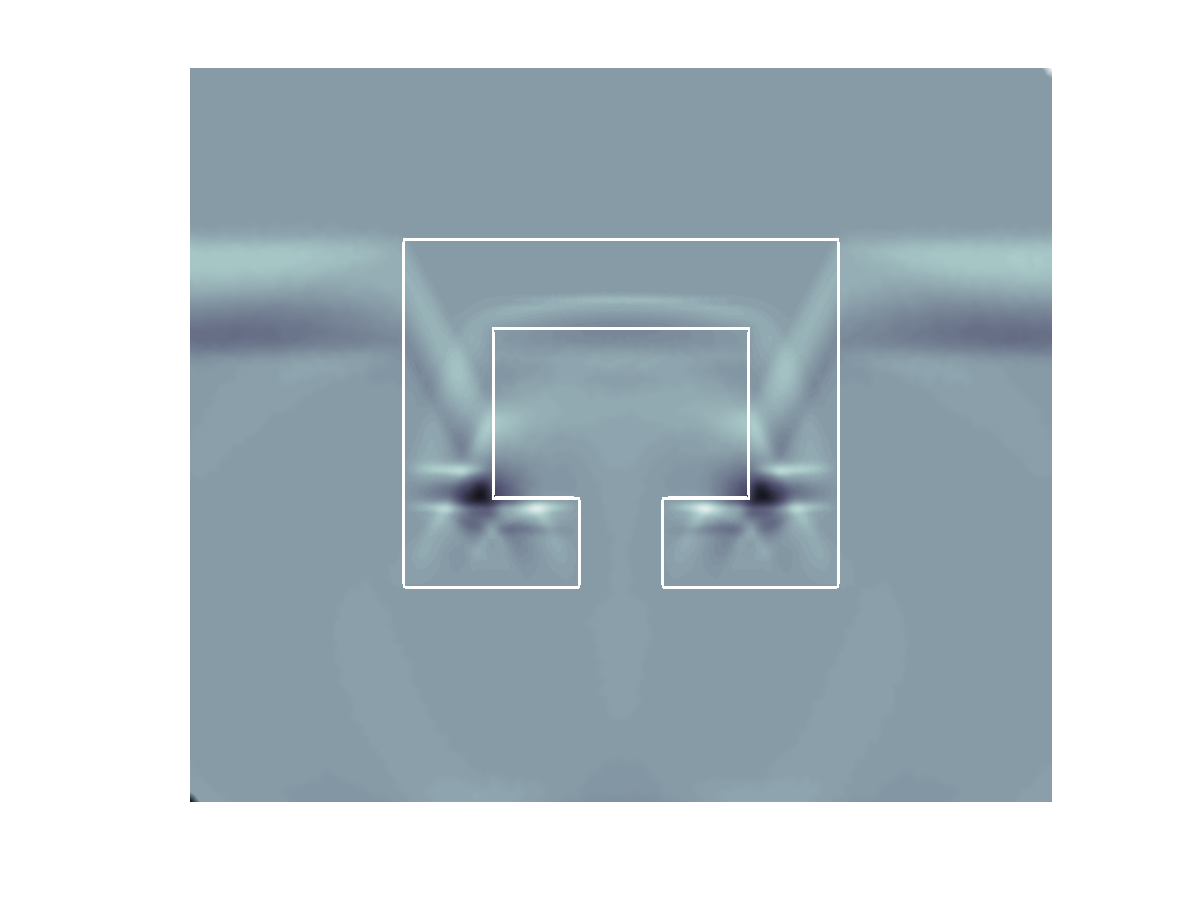}
\includegraphics[scale=.3]{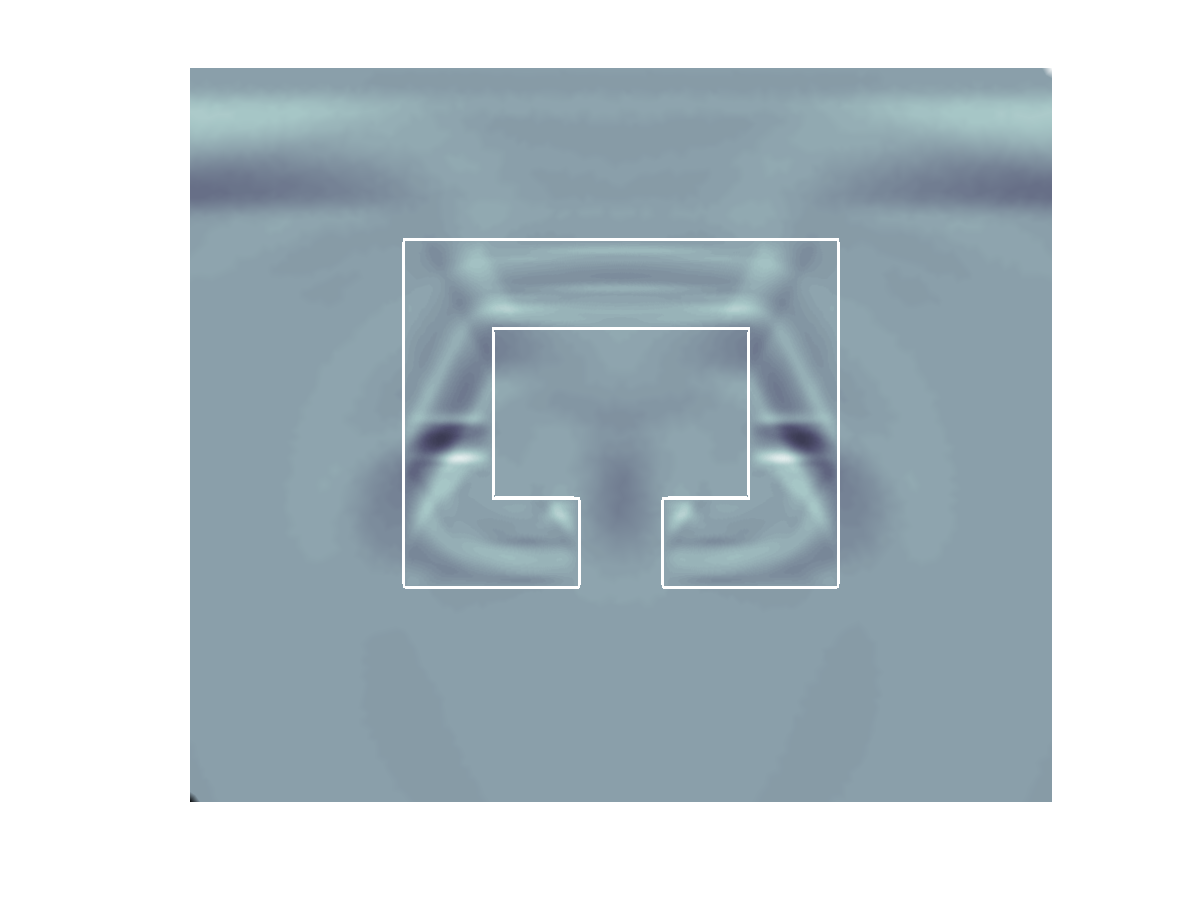}

\includegraphics[scale=.3]{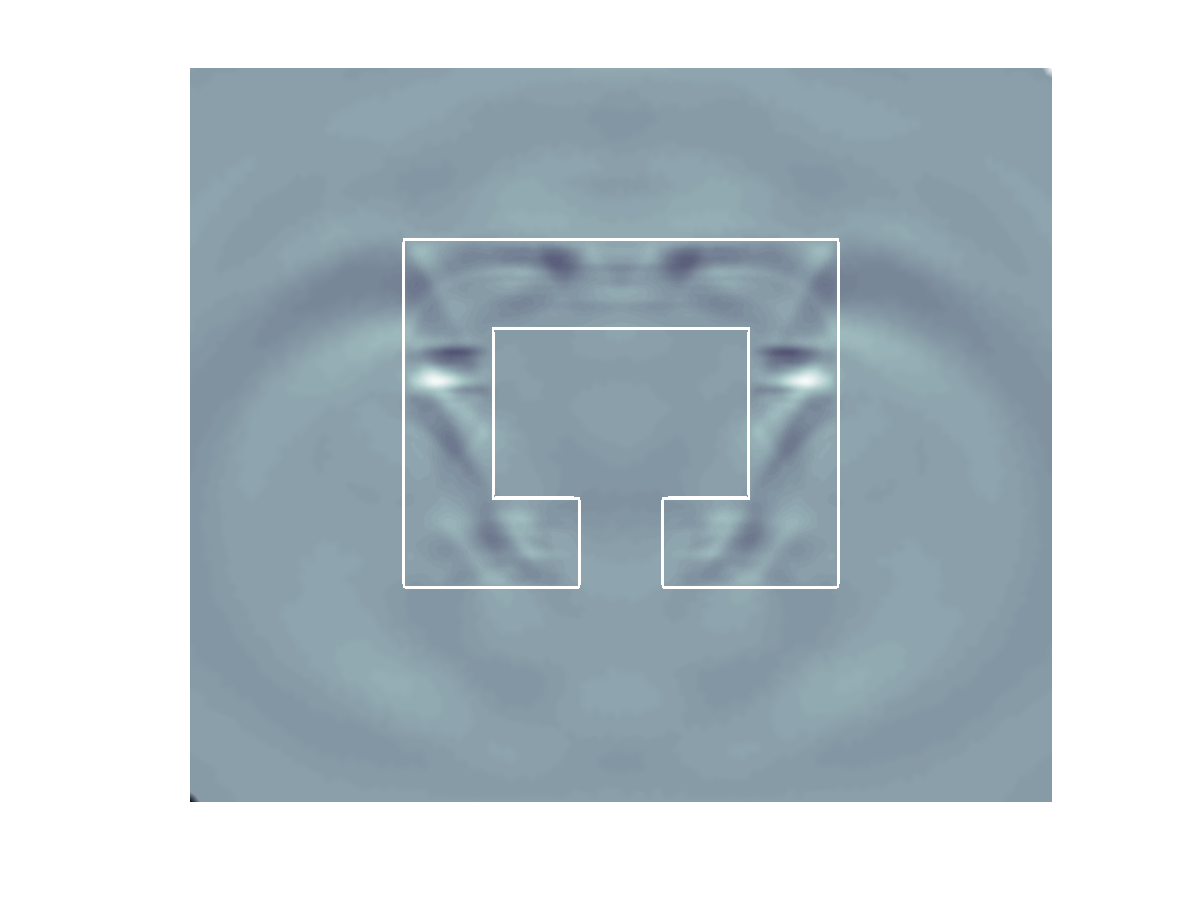}
\includegraphics[scale=.3]{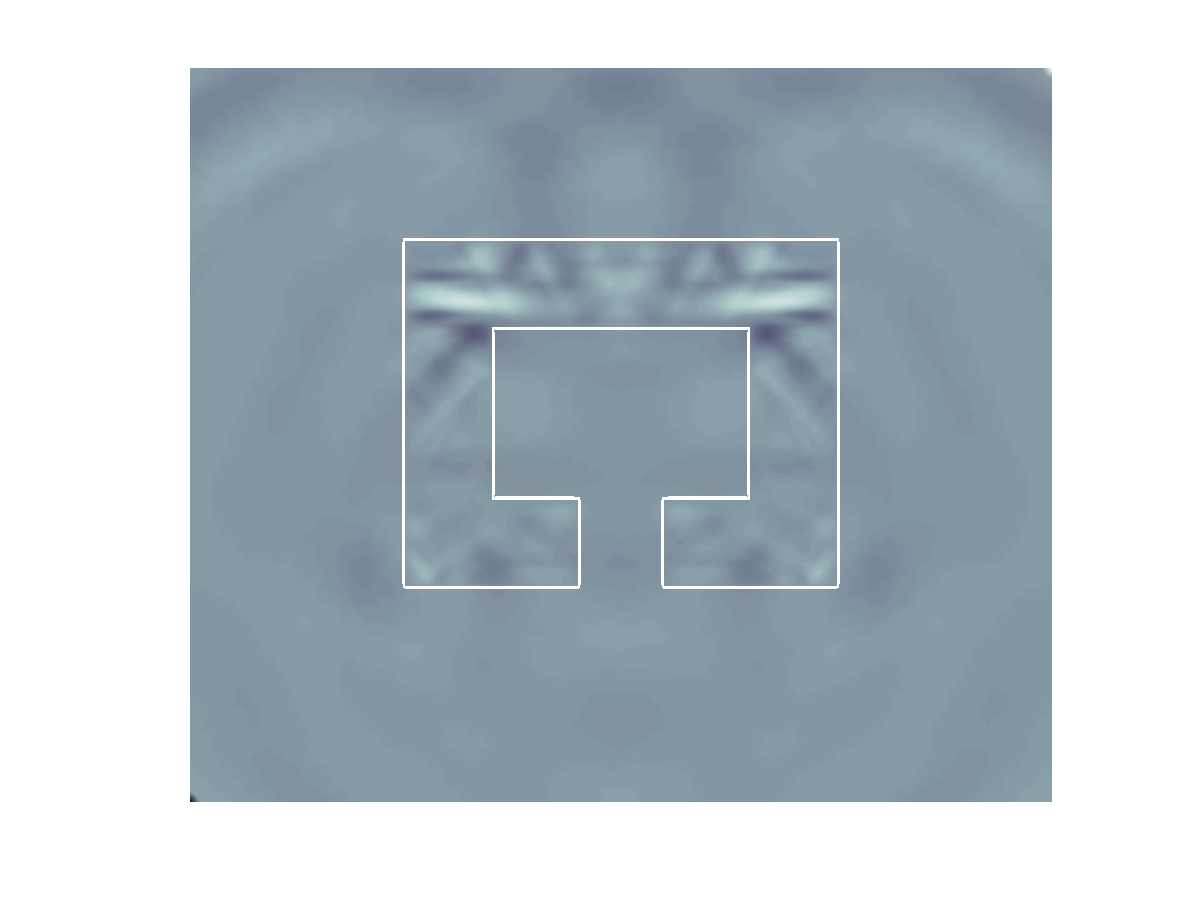}

\caption{Scattering and transmission of a wave by a non-convex domain.}\label{fig:3}
\end{center}
\end{figure}

\paragraph{Conclusions.} In this article we have presented and analyzed a fully discrete symmetric BEM-FEM scheme for transient acoustic scattering.  The analysis covers the stability of the spatial semi-discretization and convergence of a full trapezoidal rule based CQ discretization for the scattering problem.  Our theory predicts the full order of the Galerkin and CQ discretizations, which are confirmed by numerical experiments.  Similar estimates are easily derived for a Backward Euler CQ discretization, though only first order convergence will be possible.  We have explored computationally the use of RKCQ methods for time discretization, with which we are able to see convergence of order 3 when coupled with an appropriate spatial discretization.   A reduction to the boundary strategy allows for the application of parallel CQ, making (at least in two dimensions) the method faster than the associated marching-on-in-time scheme.

\bibliographystyle{abbrv}
\bibliography{biblio}

\end{document}